\documentclass[11pt, oneside, dvipsnames, svgnames, table]{amsart}
\usepackage[letterpaper,margin=3cm]{geometry}
\usepackage{sgheader}
\usepackage[T1]{fontenc}
\usepackage[utf8]{inputenc}
\usepackage{adjustbox}
\usepackage{fourier}

\usepackage[backref=true,giveninits=true,backend=biber,style=numeric,url=false,maxalphanames=4]{biblatex}
\begin{document}

\sloppy

\title{The nonexistence of sections of Stiefel varieties and stably free modules}

\author{Sebastian Gant}
\address{Department of Mathematics, the University of British Columbia\\
  1984 Mathematics Rd \\
  Vancouver BC V6T 1Z2\\
Canada.}
\email[W.~S.~Gant]{wsgant@math.ubc.ca}
\thanks{We acknowledge the support of the Natural Sciences and Engineering Research Council of Canada (NSERC), RGPIN-2021-02603.}

\begin{abstract}
  Let $V_r(\A^n)$ denote the Stiefel variety $\GL_n/\GL_{n-r}$ over a field. There is a natural projection $p: V_{r+\ell}(\A^n) \to V_r(\A^n)$. The question of whether this projection admits a section was asked by M. Raynaud in 1968. We focus on the case of $r \ge 2$ and provide examples of triples $(r,n,\ell)$ for which a section does not exist. Our results produce examples of stably free modules that do not have free summands of a given rank. To this end, we also construct a splitting of $V_2(\A^n)$ in the motivic stable homotopy category over a field, analogous to the classical stable splitting of the Stiefel manifolds due to I. M.~James.
\end{abstract}

\maketitle

\section{Introduction}

Let $\GL_n$ denote the general linear group scheme over $\Z$. Given a nonnegative integer $r$ with $r \le n$, we let $V_r(\A^n_\Z)$ denote the \emph{Stiefel scheme}: the homogeneous space $\GL_n/\GL_{n-r}$, where $\GL_{n-r}$ is embedded in $\GL_n$ via the inclusion:
\[
  A \mapsto 
  \begin{bmatrix}
    A & \\
    & I_{n-r}
  \end{bmatrix}.
\]
If $\ell$ is another nonnegative integer such that $r+\ell \le n$, there is a canonical projection $p_{r+\ell,r}^n \colon V_{r+\ell}(\A^n_\Z) \to V_r(\A^n_\Z)$ which we simply denote by $p$ when the indices are clear from context. M. Raynaud asked the following question in \cite{Raynaud1968}: for which triples $(r,n,\ell)$ does the projection $p\colon V_{r+\ell}(\A^n_\Z) \to V_r(\A^n_\Z)$ admit a section? The case of $r = \ell = 1$ is known: a section of $p\colon V_2(\A^n_\Z) \to V_1(\A^n_\Z)$ exists if and only if $n$ is even. Another known case is that of $r=n-1$ and $\ell=1$, where a section of $p$ exists over $\Z$ for all $n \ge 1$. One way of seeing this is to note that there are isomorphisms of schemes
\[
  V_n(\A^n_\Z) \iso \GL_n, \qquad V_{n-1}(\A^n_\Z) \iso \SL_n
\]
for which a section of $p$ is given by the obvious inclusion $\SL_n \hookrightarrow \GL_n$ (see also \cite[Proposition 2.2]{Raynaud1968}). 

The case of $r=1$ and $\ell > 1$ was examined over fields in \cite{Raynaud1968} and \cite{Gant2025}, where \cite{Raynaud1968} produces examples of pairs $(n,\ell)$ for which a section of $p: V_{1+\ell}(\A^n) \to V_1(\A^n)$ does not exist, and \cite{Gant2025} produces examples of pairs $(n, \ell)$ for which a section exists. Here and throughout, we denote by $V_r(\A^n)$ the base change of $V_r(\A^n_\Z)$ to a specified base field $k$. 

In this paper, we focus on the case of $r \ge 2$ over a field. To the author's knowledge, the only known answer to the above question when $r\ge 2$ (before the writing of this paper) is that of $r=n-1$ and $\ell=1$ mentioned above. We prove the following, which accounts for almost all the remaining cases.

\begin{IntroTheorem}[cf. Theorems \ref{thm:l>1} and \ref{thm:l=1}]\label{thm:Intro}
  Let $k$ be a field and $r,n,\ell$ be nonnegative integers with $r+\ell \le n$ and $r \ge 2$. There does not exist a section of $p: V_{r+\ell}(\A^n) \to V_r(\A^n)$ over $k$ in the following cases:
  \begin{enumerate}
    \item $\ell \ge 2$;
    \item $\ell=1$, $r \le n-2$, and $n-r \not\equiv 1 \pmod{24}$. 
  \end{enumerate}
\end{IntroTheorem}

In particular, there does not exist a section of $p$ over $\Z$ in the cases of Theorem \ref{thm:Intro}. We remark that, when $k$ admits a complex embedding, a short proof of Theorem \ref{thm:Intro} can be obtained using complex realization and the results of \cite{Suter1970}. We also expect that a section does not exist in the case $\ell=1$, $r \le n-2$, and $n-r \equiv 1 \pmod{24}$; indeed, we can say that a section does not exist over fields of characteristic 0 in this case by comparison to \cite{Suter1970}. 

As laid out in \cite{Raynaud1968}, the existence (or nonexistence) of a section of $p$ has consequences for the theory of stably free modules, which we now describe. To fix some terminology, let $R$ be a commutative ring and $r,n$ be nonnegative integers with $r \le n$. An $R$-module $P$ is \emph{stably free of type $(n,n-r)$} if there is an isomorphism of $R$-modules
\[
  P \oplus R^{r} \iso R^n.
\]
Over a fixed field $k$, there is a $k$-algebra $A_{n,n-r}$ and a stably free module $P_{n,n-r}$ of type $(n,n-r)$ over $A_{n,n-r}$ with the property that the projection $V_{r+\ell}(\A^n) \to V_r(\A^n)$ has a section over $k$ if and only if $P_{n,n-r}$ has a free summand of rank $\ell$ \cite[Proposition 2.4]{Raynaud1968}. In light of this observation, Theorem \ref{thm:Intro} produces examples of stably free modules of type $(n,n-r)$ that do not have free summands of a fixed rank $\ell$; we refer to Theorem \ref{thm:Main} for the precise module-theoretic statement. 

The techniques we use to prove Theorem \ref{thm:Intro} rely on the $\A^1$-homotopy theory of \cite{Morel1999}. Most of our methods are motivic analogues of the homotopy-theoretic methods in \cite{Suter1970}, which addresses the nonexistence of sections of the projections of the complex Stiefel manifolds
\[
  \U(n)/\U(n-r-\ell) \to \U(n)/\U(n-r)
\]
when $r \ge 2$. In this vein, we examine a certain map of spaces
\[
  f_r^n: \Sigma^{1,1} \tilde{\P}^{n-1}_{n-r} \to V_r(\A^n),
\]
originally constructed in \cite{Williams2012}, where the space $\Sigma^{1,1} \tilde{\P}^{n-1}_{n-r}$ is (equivalent to) the $\G_m$-suspension of the so-called \emph{truncated projective space} $\P^{n-1}/\P^{n-r-1}$. 

Our technique is first to compute the $\A^1$-connectivity of the map $f_r^n$ (Proposition \ref{prop:frnConnectivity}) using a motivic version of the Blakers--Massey Theorem (Proposition \ref{prop:MotB-M}). The connectivity calculation allows us to show that if a section of $p$ exists, then there is a retract in $K$-theory of a certain map between truncated projective spaces. The Adams operations then obstruct the existence of such a retract (see Lemmas \ref{lem:NoRetract} and \ref{lem:NoRetract2}). 

In proving Theorem \ref{thm:Intro}, we also establish a stable splitting of the Stiefel variety $V_2(\A^n)$, which is perhaps of independent interest. 

\begin{IntroTheorem}[cf. Proposition \ref{prop:r=2Split}]\label{thm:Intro2}
  Suppose $k$ is a field. The map $f_2^n: \Sigma^{1,1} \tilde{\P}^{n-1}_{n-2} \to V_2(\A^n)$ has a stable retract, inducing a splitting
  \[
    \Sigma_{\P^1}^\infty V_2(\A^n) \homeq \Sigma_{\P^1}^\infty \Sigma^{1,1}\tilde{\P}^{n-1}_{n-2} \vee S^{4n-4,2n-1}
  \]
  in the motivic stable homotopy category $\cat{SH}(k)$.
\end{IntroTheorem}
Classical analogues and generalizations of this result can be found in \cite{James1959} and \cite{Miller1985}.

\subsection{Overview}

In \Cref{sec:Setup,sec:Stiefel,sec:ProjSpaces}, we establish some preliminary lemmas including a slight improvement to motivic versions of the Blakers--Massey in the literature (see \cite[Theorem 4.2.1]{Asok2016}, \cite[Proposition 2.21]{Wickelgren2019}, \cite[Proposition 3.3]{Asok2024a}, and \cite[Theorem 2.3.8]{Strunk2012}). \Cref{sec:frn} provides a calculation of the $\A^1$-connectivity of the comparison map $f_r^n: \Sigma^{1,1} \tilde{\P}^{n-1}_{n-r} \to V_r(\A^n)$. In \Cref{sec:Splitting}, we prove the stable splitting of Theorem \ref{thm:Intro2}. The $K$-theoretic obstructions to sections between maps of truncated projective spaces are discussed in \Cref{sec:KT}. Lastly, in \Cref{sec:Sections}, we prove Theorem \ref{thm:Intro} and rephrase our results in terms of stably free modules.

\subsection{Acknowledgements}

The author would like to thank Ben Williams and Oliver R\"ondigs for helpful discussions, particularly about material in Appendix \ref{sec:ProofOfInc}, as well as the anonymous referee, whose comments strengthened many of the results of this paper.

\section{Setup}\label{sec:Setup}

Throughout, $k$ denotes a base field, and we denote by $\Sm_k$ the category of smooth, separated, finite-type $k$-schemes. We let $\Pre(\Sm_k)$ denote the $\infty$-category of presheaves of spaces on $\Sm_k$. The full subcategory of $\Pre(\Sm_k)$ spanned by the Nisnevich sheaves of spaces is denoted $\Shv_{\Nis}(\Sm_k)$, and we let $\cat{Spc}(k)$ denote the $\infty$-category of motivic spaces: the full subcategory of $\cat{Shv}_{\Nis}(\Sm_k)$ spanned by the $\A^1$-invariant Nisnevich sheaves. The inclusion $\cat{Spc}(k) \to \cat{Shv}_{\Nis}(\Sm_k)$ is an accessible localization \cite[Section 3.4]{Hoyois2017}, and the associated localization endofunctor
\[
  L_{\A^1}: \cat{Shv}_{\Nis}(\Sm_k) \to \cat{Shv}_{\Nis}(\Sm_k)
\]
is called the \emph{$\A^1$-localization functor}. A map $f: \sh{X} \to \sh{Y}$ in $\Shv_{\Nis}(\Sm_k)$ is an \emph{$\A^1$-equivalence} if $L_{\A^1}(f)$ is an equivalence. We shall only use the symbol ``$\homeq$'' to denote an $\A^1$-equivalence between objects. The $\A^1$-homotopy category, the homotopy category associated with $\Spc(k)$, is denoted $\cat{H}(k)$. The space of maps $\sh{X} \to \sh{Y}$ in $\Shv_{\Nis}(\Sm_k)$ is denoted $\Map_k(\sh{X},\sh{Y})$.

There are pointed variants of the above constructions. We let $\cat{H}(k)_*$ denote the pointed $\A^1$-homotopy category, the homotopy category associated with $\Spc(k)_*$. If $\sh{X},\sh{Y}$ are objects of $\cat{Shv}_{\Nis}(\Sm_k)_*$, we denote the set of maps $L_{\A^1}\sh{X} \to L_{\A^1}\sh{Y}$ in $\cat{H}(k)_*$ by $[\sh{X},\sh{Y}]$.  

Let $\EuScript{E}$ be an $\infty$-topos. For $n \ge -1$, we say that a morphism $f$ in $\EuScript{E}$ is \emph{$n$-connected} if $f$ is $(n+1)$-connective in the sense of \cite[Definition 6.5.1.10]{Lurie2009}. It is convenient to make the convention that any map is $-2$-connected. An object $\sh{X}$ of $\EuScript{E}$ is \emph{$n$-connected} if $\sh{X} \to *$ is $n$-connected. In the $\infty$-topos of spaces, an object $X$ is $n$-connected if and only if $\pi_i(X,x)$ vanishes for all $i \le n$ and all basepoints $x \in X$, and a morphism is $n$-connected if and only if each of its homotopy fibres is $n$-connected. 

Since the $\infty$-topos $\Shv_{\Nis}(\Sm_k)$ has enough points, a morphism in $\Shv_{\Nis}(\Sm_k)$ is $n$-connected if and only if $p^*f$ is an $n$-connected map of spaces for all points $p$ in a conservative family (see Section \ref{sec:Points} for a discussion of points). If $f:\sh{X} \to \sh{Y}$ is a morphism in $\Shv_{\Nis}(\Sm_k)$, we say that $f$ is \emph{$\A^1$-$n$-connected} if $L_{\A^1}f$ is $n$-connected. Similarly, an object $\sh{X}$ is \emph{$\A^1$-$n$-connected} if $L_{\A^1}\sh{X}$ is $n$-connected.  

\begin{remark}
  This numbering for the connectedness of an object is standard, but our numbering for the connectedness of a map differs from that of classical sources. Our numbering appears to be more common in the modern literature (see, for example, \cite[Remark 3.3.5]{Anel2020}).
\end{remark}

If $p\ge q\ge0$, we let
\[
  S^{p,q} := S^{p-q} \Smash (\G_m)^{\Smash q}
\]
where $S^{p-q}$ is the simplicial $p-q$-sphere. If $\sh{X}$ is a pointed object, we let
\[
  \Sigma^{p,q} \sh{X} := S^{p,q} \Smash \sh{X}.
\]

If $f:\sh{X} \to \sh{Y}$ is a pointed map, we use $\fib(f)$ and $\cof(f)$ to denote the fibre and cofibre of $f$ (calculated in  $\cat{Shv}_{\Nis}(\Sm_k)$). Note that if $\sh{Y}$ is connected, then $f$ is $n$-connected if and only if $\fib(f)$ is $n$-connected.

The \emph{$\A^1$-fibre of $f$}, denoted $\fib_{\A^1}(f)$, is the motivic space $\fib(L_{\A^1}f)$. The \emph{$\A^1$-cofibre of $f$}, which we denote by $\cof_{\A^1}(f)$, is the motivic space $L_{\A^1}\cof(L_{\A^1}f)$. Since the functor $L_{\A^1}$, considered as a functor $\cat{Shv}_{\Nis}(\Sm_k) \to \Spc(k)$, preserves cofibre sequences, the $\A^1$-cofibre of $f$ is naturally equivalent to both the cofibre of $L_{\A^1}f$ (calculated in $\Spc(k)$) as well as the motivic space $L_{\A^1}\cof(f)$.

\begin{lemma}\label{lem:SmashConn}
  Suppose $\sh{X},\sh{Y}$ are objects of $\cat{Shv}_{\Nis}(\Sm_k)_*$ which are pulled back from a perfect subfield of $k$. Assume that $\sh{X}$ is $\A^1$-$n$-connected and $\sh{Y}$ is $\A^1$-$m$-connected (with $n,m \ge -1$). Then the smash product $\sh{X} \Smash \sh{Y}$ is $\A^1$-$(n+m+1)$-connected. 
\end{lemma}
\begin{proof}
  This follows from Morel's Unstable $\A^1$-Connectivity Theorem \cite[Theorem 6.38]{Morel2012} and Lemma \ref{lem:BaseChange}.
\end{proof}

We will also use the following lifting lemma.

\begin{lemma}\label{lem:CellApprox}
  Let $n \geq 1$ be an integer. Suppose $f:\sh{Y} \to \sh{Z}$ is an $n$-connected map of pointed, simply connected motivic spaces over a field $k$. If $g: \sh{X} \to \sh{Z}$ is a map of pointed motivic spaces such that $\sh{X}$ has Nisnevich cohomological dimension at most $n+1$, then there is a map $\tilde{g}: \sh{X} \to \sh{Y}$ making the diagram
  \[
    \begin{tikzcd}
      & \sh{Y} \dar["f"] \\
      \sh{X} \ar[ur,"\tilde{g}"] \rar["g"] & \sh{Z}
    \end{tikzcd}
  \]
  commute. If, moreover, $\sh{X}$ has Nisnevich cohomological dimension at most $n$, then the lift $\tilde{g}$ is unique (up to homotopy).
\end{lemma}

\begin{proof}
  This is a straightforward application of the Moore--Postnikov factorization in $\A^1$-homotopy theory \cite[Theorem 6.1.1]{Asok2015}. 
\end{proof}

\subsection{A Motivic Blakers--Massey Theorem}

Heuristically, the classical Blakers--Massey Theorem asserts that, given a map of pointed spaces $f: X \to Y$, the space $\Omega \cof(f)$ is an approximation to $\fib(f)$ in a range of homotopy groups that depends on the connectivity of $X$ and $f$ (see Corollary \ref{cor:B-M2} for a precise statement).

Motivic versions of the Blakers--Massey Theorem appear in the literature as \cite[Theorem 4.2.1]{Asok2016}, \cite[Proposition 2.21]{Wickelgren2019}, \cite[Proposition 3.3]{Asok2024a}, and \cite[Theorem 2.3.8]{Strunk2012}, often under an $\A^1$-simple-connectivity hypothesis. Over a perfect field, we establish a motivic version of the Blakers--Massey Theorem (Proposition \ref{prop:MotB-M}) without this hypothesis. The key to this slight improvement is an observation of \cite{Asok2024} concerning strongly $\A^1$-invariant sheaves. The rest of the argument follows that of \cite[Theorem 4.2.1]{Asok2016} closely by appealing to results of \cite[Chapter 6]{Morel2012} concerning the $\A^1$-localization of fibre sequences. 

The following result, a Blakers--Massey theorem for $\infty$-toposes, can be found in \cite{Anel2020}. 

\begin{proposition}[{\cite[Corollary 4.3.1]{Anel2020}}]\label{prop:B-M1}
  Let $n,m \geq -1$ be integers. If 
  \[
    \begin{tikzcd}
      \sh{X} \dar["f"] \ar[dr,phantom,very near end, "\ulcorner"] \rar["g"] & \sh{Z} \dar \\
      \sh{Y} \rar & \sh{W}
    \end{tikzcd}
  \]
  is a pushout square in an $\infty$-topos such that $f$ is $n$-connected and $g$ is $m$-connected, then the natural map 
  \[
    \sh{X} \to \sh{Y} \times_{\sh{W}} \sh{Z}
  \]
  is $(n+m)$-connected.
\end{proposition}

The following important special case of Proposition \ref{prop:B-M1}, when $\sh{Z} = *$, is also called ``the Blakers--Massey Theorem'' in the literature.

\begin{corollary}\label{cor:B-M2}
  Let $n,m \geq -1$ be integers. Suppose $f:\sh{X} \to \sh{Y}$ is a map of pointed objects in an $\infty$-topos such that $\sh{X}$ is $n$-connected and $f$ is $m$-connected, then the natural map
  \[ \fib(f) \to \Omega \cof(f) \]
  is $(n+m)$-connected. 
\end{corollary}

We may now prove a motivic version of Corollary \ref{cor:B-M2} in the case that the base field is perfect.

\begin{proposition}[Motivic Blakers--Massey]\label{prop:MotB-M}
  Let $k$ be a perfect field and $n,m \ge -1$ be integers with $n+m \ge 0$. Suppose $f:\sh{X} \to \sh{Y}$ is a map of pointed Nisnevich sheaves such that $\sh{X}$ is $\A^1$-$n$-connected and $f$ is $\A^1$-$m$-connected. Then the natural map 
  \[
    \fib_{\A^1}(f) \to \Omega \cof_{\A^1}(f)
  \]
  is $(n+m)$-connected.
\end{proposition}
\begin{proof} 
  Proposition \ref{prop:B-M1} applied to the map $L_{\A^1}f$ (in the $\infty$-topos $\cat{Shv}_{\Nis}(\Sm_k)$) asserts that the natural map $\fib_{\A^1}(f) \to \Omega \cof(L_{\A^1}f)$ is $(m+n)$-connected. In particular, since $m+n \ge 0$, the induced map
  \[
    \bpi_1(\fib_{\A^1}(f)) \to \bpi_1(\Omega \cof(L_{\A^1}f))
  \]
  is an epimorphism. The sheaf of groups $\bpi_1(\fib_{\A^1}(f))$ is strongly $\A^1$-invariant by \cite[Theorem 6.1]{Morel2012}. Note also that $\bpi_1(\Omega \cof(L_{\A^1}f)) \iso \bpi_0(\Omega^2 \cof(L_{\A^1}f))$, and $\Omega^2 \cof(L_{\A^1}f)$ is an $h$-space---a grouplike monoid in the homotopy category. We may apply \cite[Lemma 2.1.11]{Asok2024} to deduce that $\bpi_1(\Omega \cof(L_{\A^1}f))$ is strongly $\A^1$-invariant itself (see also \cite[Proposition 2.8]{Asok2024a} which says that, over a perfect field $k$, the notions of strong $\A^1$-invariance and very strong $\A^1$-invariance agree). Using \cite[Theorem 6.56]{Morel2012}, we conclude that the induced map of $\A^1$-localizations
  \[ \fib_{\A^1}(f) \to L_{\A^1}(\Omega \cof(L_{\A^1}f)) \]
  is $(m+n)$-connected. Since $\bpi_1(\Omega \cof(L_{\A^1}f))$ is strongly $\A^1$-invariant, we may apply \cite[Theorem 6.46]{Morel2012} to conclude that the natural map 
  \[ L_{\A^1} \Omega \cof(L_{\A^1}f) \to \Omega \cof_{\A^1}(f) \]
  is an equivalence, establishing the proposition. 
\end{proof}

As an application, we prove the following lemma.
\begin{lemma}\label{lem:B-MApp1}
  Let $k$ be a perfect field and $n>0$ an integer. Suppose $f:\sh{X} \to \sh{Y}$ is a map of pointed Nisnevich sheaves such that $\sh{X}$ is $\A^1$-simply connected, $\sh{Y}$ is $\A^1$-connected, and $\cof_{\A^1}(f)$ is $n$-connected. Then $f$ is $\A^1$-$n-1$-connected. 
\end{lemma}
\begin{proof}
  It suffices to show that $\fib_{\A^1}(f)$ is $n-1$-connected. Proposition \ref{prop:MotB-M} applies to the map $f$, and we deduce that the natural map
  \[
    \fib_{\A^1}(f) \to \Omega \cof_{\A^1}(f)
  \]
  is connected. In particular, the induced map
  \[
    \bpi_0(\fib_{\A^1}(f)) \to \bpi_1(\cof_{\A^1}(f))=0
  \]
  is an isomorphism. If $n=1$, we are done. Otherwise, we may apply Proposition \ref{prop:MotB-M} inductively to conclude that $\fib_{\A^1}(f) \to \Omega \cof_{\A^1}(f)$ is $n-1$-connected, so that $\fib_{\A^1}(f)$ is $n-1$-connected. 
\end{proof}

\section{Stiefel varieties}\label{sec:Stiefel}
Let $\cat{CAlg}_k$ denote the category of commutative $k$-algebras. Given nonnegative integers $r,n$ with $r \leq n$, the Stiefel variety, denoted $V_r(\A^n)$, is the affine $k$-scheme representing the functor
\[
  \cat{CAlg}_k \to \cat{Set}: R \mapsto \{ (A,B) \in \Mat_{r \times n}(R) \mid AB^T = I_r \},
\]
where $I_r$ is the $r \times r$ identity matrix. We endow the Stiefel variety $V_r(\A^n)$ with a basepoint given by the $k$-rational point
\[
  \left( \begin{bmatrix} I_r & 0 \end{bmatrix}, \begin{bmatrix} I_r & 0 \end{bmatrix} \right),
\]
so we may consider $V_r(\A^n)$ as an object of $\Shv_{\Nis}(\Sm_k)_*$.
Some special instances of Stiefel varieties are
\begin{align*}
  V_0(\A^n) = \Spec k, &\qquad V_1(\A^n) = Q_{2n-1} \simeq S^{2n-1,n}, \\
  V_{n-1}(\A^n) \iso \SL_n, &\qquad V_n(\A^n) \iso \GL_n.
\end{align*}

There is a closely related $k$-scheme also termed ``the Stiefel variety'' in the literature: let $V_r'(\A^n)$ denote the $k$-scheme representing the functor
\[
  \cat{CAlg}_k \to \cat{Set} : R \mapsto \{ A \in \Mat_{r \times n}(R) \mid \exists B \in \Mat_{r \times n}(R) \text{ such that } AB^T = I_r \}.
\]
The projection onto the first factor $V_r(\A^n) \to V_r'(\A^n)$ is an affine-space bundle and hence an $\A^1$-equivalence. Our definitions are perhaps not standard, but we are content with these choices since the $k$-varieties $V_r(\A^n)$ play a more central role in this paper.

Given another nonnegative integer $\ell$ satisfying $r + \ell \le n$, there is a closed inclusion 
\begin{align*} 
  i_{r,r+\ell}^n: V_{r}(\A^{n - \ell}) &\to V_{r+\ell}(\A^n) \\
  (A,B) &\mapsto \left(
  \begin{bmatrix}
    A & 0 \\
    0 & I_\ell
  \end{bmatrix},
  \begin{bmatrix}
    B & 0 \\
    0 & I_\ell
  \end{bmatrix} \right)
\end{align*}
where $I_\ell$ is the $\ell \times \ell$ identity matrix. There is also a projection
\[ p_{r+\ell,r}^n: V_{r+\ell}(\A^n) \to V_{r}(\A^n) \] 
given by forgetting the first $\ell$ rows. When there is no risk of confusion, we will drop the indices from the notation and denote these maps $i$ and $p$, respectively. Note that $i$ and $p$ are pointed maps. The above maps fit into an $\A^1$-fibre sequence
\[
  V_r(\A^{n-\ell}) \xrar{i} V_{r+\ell}(\A^n) \xrar{p} V_r(\A^n)
\]
(see e.g. \cite[Section 3.2]{Gant2025}). We let $i'$ and $p'$ denote the analogous inclusion $V'_r(\A^{n-\ell}) \to V'_{r+\ell}(\A^n)$ and projection $V_{r+\ell}'(\A^n) \to V_r'(\A^n)$, respectively.

\subsection{Two lemmas regarding sections}
By a \emph{section} of $p:V_{r+\ell}(\A^n)\to V_r(\A^n)$, we mean a right-inverse of $p$ in the category of $k$-schemes. The following two lemmas provide useful reductions in arguments appearing later in the paper. 
\begin{lemma}\label{lem:Sec1}
  Let $r,\ell,n$ be nonnegative integers with $r+\ell \leq n$. If $p: V_{r+\ell}(\A^n) \to V_r(\A^n)$ has a section, then $p: V_{r+\ell'}(\A^n) \to V_r(\A^n)$ has a section for any nonnegative integer $\ell' \leq \ell$.
\end{lemma}
The proof is straightforward.
\begin{lemma}\label{lem:Sec2}
  Let $r,\ell,n,s$ be nonnegative integers with $r+\ell \leq n$. If $p: V_{r+s+\ell}(\A^{n+s}) \to V_{r+s}(\A^{n+s})$ has a section, then $p: V_{r+\ell}(\A^n) \to V_r(\A^n)$ has a section. 
\end{lemma}
\begin{proof}
  Let $\phi: V_{r+s}(\A^{n+s}) \to V_{r+s+\ell}(\A^{n+s})$ be a section of $p: V_{r+s+\ell}(\A^{n+s}) \to V_{r+s}(\A^{n+s})$. We claim that the composite 
  \[ V_r(\A^n) \xrar{i} V_{r+s}(\A^{n+s}) \xrar{\phi} V_{r+s+\ell}(\A^{n+s}) \] 
  factors through the inclusion $i: V_{r+\ell}(\A^n) \to V_{r+s+\ell}(\A^{n+s})$, providing a section of $p: V_{r+\ell}(\A^n) \to V_r(\A^n)$. To prove the claim, we argue on $R$-points. If $(A,B)$ is an $R$-point of $V_r(\A^n)$, then the element $\phi \circ i_{r,r+s}((A,B))$ is a pair of matrices
  \[\renewcommand{\arraystretch}{1.2}
    (A',B') = 
    \left( \left[
    \begin{array}{c|c}
      C_1 & D_1 \\ \hline
      A & 0 \\ \hline
      0 & I_s
    \end{array} \right],
    \left[
    \begin{array}{c|c}
      C_2 & D_2 \\ \hline
      B & 0 \\ \hline 
      0 & I_s
    \end{array} \right]
    \right) \in \Mat_{r+s+\ell, n+s}(R)^2
  \]
  where $C_i \in \Mat_{\ell \times n}(R)$ and $D_i \in \Mat_{\ell \times s}(R)$ for $i=1,2$. The condition $A'B'^T=I_{r+s+\ell}$ has the following implications:
  \[
    D_1 = D_2 = 0, \qquad C_1C_2^T = I_\ell, \qquad AC_2^T=0, \qquad C_1B^T=0.
  \]
  Since the map $\phi \circ i_{r,r+s}$ on points is natural in $R$, so is the assignment
  \[
    (A,B) \mapsto \left( 
    \begin{bmatrix}
      C_1 \\ 
      A
    \end{bmatrix}, \begin{bmatrix}
      C_2 \\
      B
    \end{bmatrix} \right),
  \]
  providing a morphism of $k$-schemes $V_r(\A^n) \to V_{r+\ell}(\A^n)$ which is a section of $p_{r+\ell,r}^n$.
\end{proof}

\section{Truncated projective spaces}\label{sec:ProjSpaces}

If $n \geq -1$ is an integer, we let $\tilde{\P}^n$ denote the Jouanolou device for $\P^n$, which we now describe (see also \cite[pp. 69]{Voevodsky2003a}, \cite[Section 5]{Williams2012}). Let $R$ be a commutative $k$-algebra. An \emph{$(n+1)$-generated split line bundle over $R$} is a pair of maps of $R$-modules
\[
  (q:R^{n+1} \to L, s: L \to R^{n+1})
\]
where $L$ is a projective $R$-module of rank $1$, the map $q$ is an $R$-module epimorphism, and $q \circ s=\id_{L}$. Note that such a pair $(q,s)$ entails an isomorphism of $R$-modules $R^{n+1} \iso \ker q \oplus L$. An \emph{isomorphism of $(n+1)$-generated split line bundles} from $(q,s)$ to 
\[
  (q':R^{n+1} \to L', s':L' \to R^{n+1})
\]
is an $R$-module isomorphism $f: L \to L'$ satisfying $f \circ q = q'$ and $s' \circ f = s$. We let $[q,s]$ denote the isomorphism class of $(q,s)$. 

Using the representability criterion of \cite[Theorem VI-14]{Eisenbud2000}, one may check that the functor $\cat{CAlg}_k \to \cat{Set}$ that sends a $k$-algebra $R$ to the set of isomorphism classes of $(n+1)$-generated split line bundles over $R$ is represented by a $k$-scheme which we denote $\tilde{\P}^n$.
\begin{remark}\label{rem:FPoints}
  If $E/k$ is a field extension, the $E$-points of $\tilde{\P}^n$ are in natural bijection with pairs $(L,W)$ of $E$-vector subspaces of $E^{n+1}$ such that $L$ is a line, $W$ is a hyperplane, and the $E$-vector space sum $L+W$ is direct. 
\end{remark}
\begin{remark}\label{rem:ProjWE}
  There is a projection $\pi: \tilde{\P}^n \to \P^n$ which, on $R$-points, forgets the section $s$. By considering the standard affine cover of $\P^n$ given on geometric points by 
  \[ U_i = \{ [x_0,\dots, x_n] \mid x_i \neq 0 \}, \]
  it is straightforward to check that $\pi$ is a Zariski-locally trivial bundle of affine spaces. In particular, the map $\pi$ is an $\A^1$-equivalence.     
\end{remark}

When $m,n$ are integers satisfying $-1 \leq m \leq n$, there are closed inclusions $\tilde{\imath}_{m,n}: \tilde{\P}^m \to \tilde{\P}^n$ defined on $R$-points as follows. Given an $R$-point $[q,s]$ of $\tilde{\P}^m$ represented by a pair 
\[ (q:R^{m+1} \to L,\: s: L \to R^{m+1}), \] 
we define $\tilde{\imath}_{m,n}([q,s])$ to be the isomorphism class of the pair 
\[
  (q \circ \pr: R^{n+1} \to L, \: \inc \circ s: L \to R^{n+1}), 
\]
where $\pr: R^{n+1} \to R^{m+1}$ is the projection onto the first $m+1$ factors, and $\inc: R^{m+1} \to R^{n+1}$ is the inclusion into the first $m+1$ factors. This assignment defines a morphism of $k$-schemes.

Similarly, let $\imath_{m,n}: \P^m \to \P^n$ denote the closed inclusion given on geometric points by 
\[
  [x_0,x_1, \dots, x_m] \mapsto [x_0,x_1, \dots, x_m,0, \dots, 0].
\]
We will omit subscripts and denote the above maps $\tilde{\imath}, \imath$ when there is no risk of confusion. There is a commutative diagram
\begin{equation}\label{eq:PIncs}
  \begin{tikzcd}
    \tilde{\P}^m \rar["\tilde{\imath}"] \dar["\pi","\weq"'] & \tilde{\P}^n \dar["\pi","\weq"'] \\
    \P^m \rar["\imath"] & \P^n
  \end{tikzcd}
\end{equation}
where the vertical maps are the $\A^1$-equivalences of Remark \ref{rem:ProjWE}. We consider $\P^n,\tilde{\P}^n$ to be pointed objects with basepoints given by $\imath_{0,n},\tilde{\imath}_{0,n}$, respectively.

\begin{lemma}\label{lem:iConn}
  If $m,n$ are nonnegative integers with $m \leq n$, the inclusion $\imath: \P^m \to \P^n$ is $\A^1$-$m-1$-connected. 
\end{lemma}
\begin{proof}
  Let $\bar{0}$ denote the closed point at the origin of $\A^n$. Using \cite[Theorem 6.53]{Morel2012}, the action of $\G_m$ on $\A^n \sm \bar{0}$ and $\A^m \sm \bar{0}$ yields a map of $\A^1$-fibre sequences:
  \[
    \begin{tikzcd}
      \A^{m+1} \sm \bar{0} \dar["\imath'"] \rar & \P^m \dar["\imath"] \rar & B \G_m \dar[equals]\\
      \A^{n+1} \sm \bar{0} \rar & \P^{n} \rar & B \G_m
    \end{tikzcd}
  \]
  where $\imath': \A^{m+1} \sm \bar{0} \to \A^{n+1} \sm \bar{0}$ is the inclusion 
  \[ (x_1, \dots, x_{m+1}) \mapsto (x_1, \dots, x_{m+1},0, \dots, 0), \]
  which is well known to be $\A^1$-$m-1$-connected \cite[Corollary 6.43]{Morel2012}. Standard results in homotopy theory yield an equivalence $\fib_{\A^1}(\imath') \to \fib_{\A^1}(\imath)$ from which we deduce the result.
\end{proof}

Let $\P_{m+1}^n, \tilde{\P}_{m+1}^n$ denote the $\A^1$-cofibres of $\imath_{m,n},\tilde{\imath}_{m,n}$, respectively (intuitively, $\P_{m+1}^n \homeq \P^n / \P^m$). Note that \eqref{eq:PIncs} induces an equivalence $\tilde{\P}_{m+1}^n \homeq \P_{m+1}^n$. We denote by $\rho$ the natural map $\P^n \to \P^n_{m+1}$ fitting into the $\A^1$-cofibre sequence 
\[
  \P^m \xrar{\imath} \P^n \xrar{\rho} \P^n_{m+1}.
\]
If $r \ge -1$ is another integer with $r \leq m$, the commuting square
\begin{equation}\label{eq:iComm}
  \begin{tikzcd}
    \P^r \rar[equals] \dar["\imath"] & \P^r \dar["\imath"] \\
    \P^m \rar["\imath"] & \P^n
  \end{tikzcd}
\end{equation}
induces a map $\P^m_{r+1} \to \P^n_{r+1}$ after taking $\A^1$-cofibres of the vertical inclusions. In a mild abuse of notation, we denote this induced map also by $\imath: \P^m_{r+1} \to \P^n_{r+1}$. Note that the commuting square \eqref{eq:iComm} induces an equivalence 
\[ \P^n_{m+1} \to \cof_{\A^1}(\imath: \P^m_{r+1} \to \P^n_{r+1}), \]
and we obtain a map $\P^n_{r+1} \to \P^n_{m+1}$ which we also denote by $\rho$. There are maps $\tilde{\imath}: \tilde{\P}^m_{r+1} \to \tilde{\P}^n_{r+1}$ and $\tilde{\rho}: \tilde{\P}^n_{r+1} \to \tilde{\P}^n_{m+1}$, defined similarly.
\begin{lemma}\label{lem:iConn2}
  Let $r,m,n$ be nonnegative integers with $r \le m \le n$. The following hold:
  \begin{enumerate}
    \item\label{i:1} The motivic space $\P^m_{r+1}$ is $\A^1$-$r$-connected;
    \item\label{i:2} The map $\imath: \P^m_{r+1} \to \P^n_{r+1}$ is $\A^1$-$m-1$-connected.
  \end{enumerate} 
\end{lemma}
\begin{proof}
  The motivic space $\P^m_{r+1}$ and the map $\imath$ may be constructed over any prime field (and in fact over $\Z$). Using Lemma \ref{lem:BaseChange}, we may therefore assume $k$ is a (perfect) prime field. First, we establish \eqref{i:1}. Since $\P^m$ is $\A^1$-connected, the statement is true if $r=0$. Suppose $r > 0$, and note that the inclusion
  \[ \imath_{r,m}: \P^r \to \P^m \]
  is a $\A^1$-$r-1$-connected by Lemma \ref{lem:iConn}. Since $r > 0$, we may apply Proposition \ref{prop:MotB-M} to the map $\imath_{r,m}$ to deduce that the natural map
  \[
    \fib_{\A^1}(\imath_{r,m}) \to \Omega \P^m_{r+1}
  \]
  is $r-1$-connected. In particular, the induced map 
  \[ 
    \bpi_i(\fib_{\A^1}(\imath_{r,m})) \to \bpi_i(\Omega \P^m_{r+1}) \cong \bpi_{i+1}(\P^m_{r+1})
  \]
  is an isomorphism for $i < r$. The homotopy sheaves $\bpi_i(\fib_{\A^1}(\imath_{r,m}))$ are trivial when $i < r$, so we are done.

  For \eqref{i:2}, the case of $r=0$ is handled by Lemma \ref{lem:iConn}, so assume $r > 0$. It suffices to treat the case of $n=m+1$, as the map $\imath: \P_{r+1}^m \to \P_{r+1}^n$ factors as
  \[ 
    \P_{r+1}^m \xrar{\imath} \P_{r+1}^{m+1} \xrar{\imath} \cdots \xrar{\imath} \P_{r+1}^{n-1} \xrar{\imath} \P_{r+1}^n. 
  \]
  The $\A^1$-cofibre of $\imath: \P^m_{r+1} \to \P^{m+1}_{r+1}$ is equivalent to the $\A^1$-cofibre of the inclusion $\imath: \P^m \to \P^{m+1}$, which is the motivic sphere $S^{2(m+1),m+1}$.  
  
  The source and target of $\imath: \P^m_{r+1} \to \P^{m+1}_{r+1}$ are $\A^1$-$r$-connected by \eqref{i:1}, so $\imath$ is (at least) $\A^1$-$r-1$-connected. Since $r > 0$, we may apply Lemma \ref{lem:B-MApp1} to $\imath: \P^m_{r+1} \to \P^{m+1}_{r+1}$, noting that $S^{2(m+1),m+1}$ is $\A^1$-$m$-connected \cite[Corollary 6.43]{Morel2012}, to conclude that $\imath$ is $\A^1$-$m-1$-connected. 
\end{proof}

\section{The comparison map \texorpdfstring{$f_r^n: \Sigma^{1,1} \tilde{\P}_{n-r}^{n-1} \to V_r(\A^n)$}{frn}}\label{sec:frn}

We now describe the comparison map $f_r^n: \Sigma^{1,1} \tilde{\P}_{n-r}^{n-1} \to V_r(\A^n)$ which plays a central role in calculations to come. The map $f_r^n$ was originally constructed in \cite{Williams2012}; we review the construction.

An $R$-point of the $k$-scheme $\tilde{\P}^{n-1} \times \G_m$ is given by a pair $([q,s], \lambda)$, where $[q,s]$ is an $R$-point of $\tilde{\P}^{n-1}$ represented by  
\[
  (q: R^n \to L,\: s: L \to R^n),
\]
and $\lambda \in R^\times$. There is a morphism of $k$-schemes
\[ f'_n : \tilde{\P}^{n-1} \times \G_m \to \GL_n, \]
which on $R$-points sends the pair $([q,s],\lambda)$ to the unique element of $\GL_n(R)$ which, under the isomorphism $R^n \iso \ker q \oplus L$ determined by $q$ and $s$, scales $L$ by $\lambda$ and acts by the identity on $\ker q$.

The morphisms $f'_n$ are compatible with inclusions in the sense that 
\[
  \begin{tikzcd}
    \tilde{\P}^{m-1} \times \G_m \rar["\tilde{\imath} \times \id"] \dar["f'_m"] & \tilde{\P}^{n-1} \times \G_m \dar["f'_n"] \\
    \GL_m \rar["i"] & \GL_n
  \end{tikzcd}
\]
commutes. Since $f'_n$ is constant on $\tilde{\P}^{n-1} \times \{1\}$, the map $f_n'$ factors through the cofibre  
\[ \tilde{\P}^{n-1} \times \G_m / (\tilde{\P}^{n-1} \times \{1\}) \homeq \Sigma^{1,1}\tilde{\P}^{n-1}_+, \]
and we denote the resulting map
\[ f_n: \Sigma^{1,1} \tilde{\P}^{n-1}_+ \to \GL_n. \]

The left square in
\begin{equation}\label{eq:MapOnCofib}
  \begin{tikzcd}[column sep = large]
     \Sigma^{1,1}\tilde{\P}^{n-r-1}_+ \rar["\Sigma^{1,1} \tilde{\imath}_+"] \dar["f_{n-r}"] & \Sigma^{1,1}\tilde{\P}^{n-1}_+ \dar["f_n"] \rar["\Sigma^{1,1}\tilde{\rho}_+"] \rar & \Sigma^{1,1} \tilde{\P}^{n-1}_{n-r} \dar[dashed,"f_r^n"] \\
     \GL_{n-r} \rar["i"] & \GL_n \rar["p"] & V_r(\A^n)
  \end{tikzcd}
\end{equation}
commutes, where the top row is an $\A^1$-cofibre sequence, and the bottom row is an $\A^1$-fibre sequence. Since the composite
\[
  \Sigma^{1,1}\tilde{\P}^{n-r-1}_+ \xrar{f_{n-r}} \GL_{n-r} \xrar{i} \GL_n \xrar{p} V_r(\A^n)
\]
is null, the dashed arrow in \eqref{eq:MapOnCofib} exists making the right square commute. We denote this induced map $f_r^n: \Sigma^{1,1} \tilde{\P}^{n-1}_{n-r} \to V_r(\A^n)$. 
\begin{remark}
  We are abusing notation here; in general, the map $f_r^n$ only exists after $\A^1$-localization. We will continue this abuse of notation for the rest of the paper.  
\end{remark}
\begin{remark}
  The map $f_{n-1}^n: \Sigma^{1,1} \tilde{\P}^{n-1} \to V_{n-1}(\A^n) \iso \SL_n$ is studied in \cite{Rondigs2023} (where it is denoted $\psi_n$). In particular, $f_1^2: \Sigma^{1,1} \tilde{\P}^1 \to \SL_2$ is an $\A^1$-equivalence (\cite[Lemma 4.1]{Rondigs2023}, c.f. \eqref{i:1} of Proposition \ref{prop:frnConnectivity}), and the $\A^1$-cofibre of $f_2^3: \Sigma^{1,1} \tilde{\P}^2 \to \SL_3$ is equivalent to the motivic sphere $S^{8,5}$ (\cite[Lemma 4.4]{Rondigs2023}, c.f. Remark \ref{rem:Sphere}).
\end{remark}

The maps $f_r^n$ are compatible with inclusions, i.e., the diagram 
\begin{equation}\label{eq:iotai}
  \begin{tikzcd}
    \Sigma^{1,1} \tilde{\P}_{n-r}^{n-1} \dar["f_{r}^{n}"] \rar["\Sigma^{1,1} \tilde{\imath}"] & \Sigma^{1,1} \tilde{\P}_{n-r}^{n+\ell-1} \dar["f_{r+\ell}^{n+\ell}"] \\
    V_{r}(\A^{n}) \rar["i"] & V_{r+\ell}(\A^{n+\ell})
  \end{tikzcd}
\end{equation}
commutes. It follows that the maps $f_r^n$ are also compatible with the projections: the diagram 
\begin{equation}\label{eq:prho}
  \begin{tikzcd}
    \Sigma^{1,1}\tilde{\P}^{n-1}_{n-r-\ell} \dar{f_{r+\ell}^n} \rar{\Sigma^{1,1}\tilde{\rho}} & \Sigma^{1,1}\tilde{\P}^{n-1}_{n-r} \dar{f_r^n} \\
    V_{r+\ell}(\A^n) \rar{p} & V_r(\A^n)
  \end{tikzcd}
\end{equation}
commutes.

There is an $\A^1$-cofibre sequence
\[
  V_{r-1}(\A^{n-1}) \xrar{i} V_r(\A^n) \to \Sigma^{2n-1,n} V_{r-1}(\A^{n-1})_+
\]
constructed in \cite[Section 3.2]{Peng2023} (see also \cite[Proposition 4.2]{Rondigs2023}), and the maps $f_r^n$ induce a map of $\A^1$-cofibre sequences
\begin{equation}\label{eq:psiDef}
  \begin{tikzcd}
    \Sigma^{1,1} \tilde{\P}_{n-r}^{n-2} \dar["f_{r-1}^{n-1}"] \rar["\Sigma^{1,1} \tilde{\imath}"] & \Sigma^{1,1} \tilde{\P}_{n-r}^{n-1} \dar["f_r^n"] \rar & S^{2n-1,n} \dar[dashed, "\psi"]\\
    V_{r-1}(\A^{n-1}) \rar["i"] & V_r(\A^n) \rar & \Sigma^{2n-1,n}V_{r-1}(\A^{n-1})_+
  \end{tikzcd}
\end{equation}
where we denote the induced map of $\A^1$-cofibres by $\psi$. The following lemma is useful in our calculation of the $\A^1$-connectivity of $f_r^n$. 

\begin{lemma}\label{lem:IncBasepoint}
  The map $\psi$ is (up to homotopy) given by $\Sigma^{2n-1,n}(-)_+$ applied to the inclusion of the basepoint. 
  \[ \Spec k \to V_{r-1}(\A^{n-1}). \] 
\end{lemma}

Our proof of Lemma \ref{lem:IncBasepoint} is technical and will be deferred to Appendix \ref{sec:ProofOfInc}.

\subsection{The connectivity of \texorpdfstring{$f_r^n$}{frn}} 

We now come to the main result of this section. Proposition \ref{prop:frnConnectivity} may be seen as a motivic analogue of the classical fact that the real Stiefel manifold $\orth(n)/ \orth(n-r)$ admits a cell structure obtained from $\RP^{n-1}/\RP^{n-r-1}$ by attaching cells of dimension $2(n-r)+1$ and larger \cite{Whitehead1944}.

\begin{proposition}\label{prop:frnConnectivity}
  Let $n$ be a positive integer. The following hold: 
  \begin{enumerate}
    \item\label{i:a} The map $f_1^n: \Sigma^{1,1} \tilde{\P}_{n-1}^{n-1} \to V_1(\A^n)$ is an $\A^1$-equivalence;
    \item\label{i:b} Let $r$ be another integer such that $1 < r \leq n-2$. The map $f_r^n\co \Sigma^{1,1} \tilde{\P}_{n-r}^{n-1} \to V_r(\A^n)$ is $\A^1$-$2(n-r)-1$-connected.
  \end{enumerate} 
\end{proposition} 
\begin{proof}
  The construction of $f_r^n$ may be carried out over any prime field (and in fact over $\Z$), so, using Lemma \ref{lem:BaseChange}, we may assume $k$ is a (perfect) prime field. The proof of \eqref{i:b} is by induction on $r$, the base case being implied by \eqref{i:a}. There is a map of $\A^1$-cofibre sequences
  \[
    \begin{tikzcd}
      \Spec(k) = \Sigma^{1,1} \tilde{\P}^{n-2}_{n-1} \dar["f_0^{n-1}"] \rar["\Sigma^{1,1} \tilde{\imath}"] & \Sigma^{1,1} \tilde{\P}^{n-1}_{n-1} \dar["f_1^n"] \rar["\weq"] & S^{2n-1,n} \dar["\psi","\weq"'] \\
      \Spec(k) = V_0(\A^{n-1}) \rar["i"] & V_1(\A^n) \rar["\weq"] & \Sigma^{2n-1,n} \Spec(k)_+
    \end{tikzcd}
  \]
  where the right vertical map is an $\A^1$-equivalence by Lemma \ref{lem:IncBasepoint}. The 2-out-of-3 property of equivalences establishes \eqref{i:a}. 
  
  For the inductive step, we consider the diagram
  \begin{equation}\label{eq:InductiveStep}
    \begin{tikzcd}
      \Sigma^{1,1} \tilde{\P}^{n-2}_{n-r} \dar["f_{r-1}^{n-1}"] \rar["\Sigma^{1,1} \tilde{\imath}"] & \Sigma^{1,1} \tilde{\P}^{n-1}_{n-r} \dar["f_r^n"] \rar & S^{2n-1,n} \dar["\psi"] \\
      V_{r-1}(\A^{n-1}) \dar \rar["i"] & V_r(\A^n) \dar \rar & \Sigma^{2n-1,n} V_{r-1}(\A^{n-1})_+ \dar \\
      \cof_{\A^1}(f_{r-1}^{n-1}) \rar["g"] & \cof_{\A^1}(f_r^n) \rar & \Sigma^{2n-1,n} V_{r-1}(\A^{n-1})
    \end{tikzcd}
  \end{equation}
  where the top two rows and left two columns are $\A^1$-cofibre sequences and $g$ is the natural induced map. The rightmost column is an $\A^1$-cofibre sequence by Lemma \ref{lem:IncBasepoint}. Moreover, using standard results in homotopy theory, there is an $\A^1$-equivalence $\cof_{\A^1}(g) \homeq \Sigma^{2n-1,n} V_{r-1}(\A^{n-1})$. That is, every row and column of \eqref{eq:InductiveStep} is an $\A^1$-cofibre sequence.
  
  The map $f_{r-1}^{n-1}$ is $\A^1$-$(2(n-r)-1)$-connected by hypothesis, and $\Sigma^{1,1} \tilde{\P}_{n-r}^{n-2}$ is $\A^1$-$(n-r-1)$-connected by Lemmas \ref{lem:iConn2} and \ref{lem:SmashConn}. Proposition \ref{prop:MotB-M} tells us that the natural map 
  \[
    \fib_{\A^1}(f_{r-1}^{n-1}) \to \Omega \cof_{\A^1}(f_{r-1}^{n-1})
  \]
  is $(3(n-r)-2)$-connected. Since $\fib_{\A^1}(f_{r-1}^{n-1})$ is $(2(n-r)-1)$-connected and $3(n-r)-1 > 2(n-r)-1$ (as $n-r \geq 2$), we see that $\Omega \cof_{\A^1}(f_{r-1}^{n-1})$ is $(2(n-r)-1)$-connected as well, or, equivalently, that $\cof_{\A^1}(f_{r-1}^{n-1})$ is $2(n-r)$-connected. 
  
  Next, we aim to apply Lemma \ref{lem:B-MApp1} to the map $g$. Before doing so, we must verify that $g$ is $-1$-connected (i.e., that $g$ induces an epimorphism on $\pi_0(-)$). To this end, it suffices to check that $\cof_{\A^1}(f_r^n)$ is connected. The space $\Sigma^{1,1} \tilde{\P}_{n-r}^{n-1}$ is $\A^1$-$(n-r-1)$-connected, and $V_r(\A^n)$ is $\A^1$-$(n-r-1)$-connected \cite[Proposition 3.2]{Gant2025}. Since $n-r \geq 2$, the map $f_r^n$ is (at least) $\A^1$-connected. Proposition \ref{prop:MotB-M} applied to $f_r^n$, noting that $\fib_{\A^1}(f_r^n)$ is connected, tells us that $\Omega \cof_{\A^1}(f_r^n)$ is connected which implies that $\cof_{\A^1}(f_r^n)$ is connected as well. 
  
  We may now apply Lemma \ref{lem:B-MApp1} to the map $g$. The space $\Sigma^{2n-1,n} V_{r-1}(\A^{n-1})$ is $\A^1$-$(2n-r-2)$-connected by Lemma \ref{lem:SmashConn} and \cite[Proposition 3.2]{Gant2025}. Lemma \ref{lem:B-MApp1} applied to the map $g$ asserts that $g$ is $(2n-r-3)$-connected. Since $\cof_{\A^1}(f_{r-1}^{n-1})$ is $2(n-r)$-connected and $r \ge 2$ (so that $2n-r-2 \ge 2(n-r)$), we have that $\cof_{\A^1}(f_r^n)$ is $2(n-r)$-connected as well. Lemma \ref{lem:B-MApp1} now applied to the map $f_r^n$, which applies since $\Sigma^{1,1} \tilde{\P}_{n-r}^{n-1}$ is $\A^1$-simply connected and $f_r^n$ is $\A^1$-connected, establishes the result.  
\end{proof}
\begin{remark}\label{rem:Sphere}
  When $n > 1$ there is an $\A^1$-equivalence $V_1(\A^{n-1})) \homeq S^{2n-3,n-1}$. It follows from \eqref{i:a} and the inductive step when $r=2$ that the $\A^1$-cofibre of $f_2^n: \Sigma^{1,1} \tilde{\P}^{n-1}_{n-2} \to V_2(\A^n)$ is $\A^1$-equivalent to the motivic sphere $S^{4n-4,2n-1}$.
\end{remark}

\section{A stable splitting of \texorpdfstring{$V_2(\A^n)$}{V2(An)}}\label{sec:Splitting}

This section aims to show that the map 
\[
  f_2^n: \Sigma^{1,1} \tilde{\P}^{n-1}_{n-2} \to V_2(\A^n)
\]
has a retract in $\cat{SH}(k)$. The argument given here is an adaptation of I. M.~James's original argument in \cite{James1959}, which establishes the analogous result for the Stiefel manifolds in topology.

\subsection{The intrinsic join in motivic homotopy theory} 

A tool we will use to construct a stable retract of $f_2^n$ is a motivic version of James's intrinsic join, which we now describe. Recall that the join of $\sh{X}$ and $\sh{Y}$, denoted $\sh{X}*\sh{Y}$, is the pushout of the span
\[
  \begin{tikzcd}
    \sh{X} & \sh{X} \times \sh{Y} \lar["\pr_1"'] \rar["\pr_2"] & \sh{Y}.
  \end{tikzcd}
\]
Fix nonnegative integers $r,n,s,m$ with $r \le n$ and $s \le m$. Let $V_{r|s}'(\A^{n|m}) \subseteq \Mat_{r\times n} \times \Mat_{s\times m}$ denote the union of the two open subschemes $V_r'(\A^n) \times \Mat_{s \times m}$ and $\Mat_{r \times n} \times V_s'(\A^m)$, whose intersection is $V_r'(\A^n) \times V_s'(\A^m)$. This Zariski cover gives rise to a pushout square
\begin{equation}\label{eq:StiefelJoin}
  \begin{tikzcd}
    V_r'(\A^n) \times V_s'(\A^m) \rar \dar & V_r'(\A^n) \times \Mat_{s \times m} \dar \\
    \Mat_{r \times n} \times V_s'(\A^m) \rar & V_{r|s}'(\A^{n|m})
  \end{tikzcd}
\end{equation}
so the $k$-scheme $V_{r|s}'(\A^{n|m})$ is a model for the join $V_r(\A^n)*V_s(\A^m)$. When $r=s$, there is an obvious inclusion $V_{r|r}'(\A^{n|m}) \hookrightarrow V_r'(\A^{n+m})$, so we obtain a map $h: V_r(\A^n)*V_s(\A^m) \to V_r(\A^{n+m})$, well-defined up to homotopy, which we call the \emph{intrinsic join}. 

\begin{remark}\label{rem:hr=1}
  In the case $r=1$, the inclusion $V_{1|1}'(\A^{n|m}) \hookrightarrow V_1'(\A^{n+m}) = \A^{n+m} \sm \bar{0}$ is an equality of schemes, so $h$ is an $\A^1$-equivalence in this case. 
\end{remark}
\begin{remark}\label{rem:JoinCohoDim}
  The $k$-scheme $V_{r|s}'(\A^{n|m})$ has Krull dimension $rn+sm$, so the join $V_r(\A^n)*V_s(\A^m)$ has Nisnevich cohomological dimension at most $rn+sm$. 
\end{remark}

\subsection{The intrinsic join and motivic cohomology}

The motivic cohomology of the Stiefel varieties $V_r(\A^n)$ is calculated in \cite{Williams2012}. Let $\m$ denote the motivic cohomology ring $\h^{*,*}(\Spec k,\Z)$. There is a presentation 
\begin{equation}\label{eq:CohoPres}
  \h^{*,*}(V_r(\A^n),\Z) = \m[\alpha_{n-r+1}, \dots, \alpha_n]/ \mathfrak{a}, \quad |\alpha_i| = (2i-1,i)
\end{equation}
where $\mathfrak{a}$ is the ideal generated by the elements $\alpha_i^2 - \{-1\}\alpha_{2i-1}$ for each $i$ with $2i-1 \leq n$, as well as the relations imposed by motivic cohomology being graded commutative in the first grading and commutative in the second. Here, $\{-1\}$ is the element $-1 \in \m^{1,1} = k^\times$. Note that $\h^{*,*}(V_r(\A^n),\Z)$ is free and finitely generated as a module over $\m$, generated by the products
\[
  \alpha_{i_1} \cdots \alpha_{i_\ell}, \qquad n-r+1 \leq i_1 < \cdots < i_{\ell} \leq n.
\]
We abuse notation and also denote the image of the algebra generators in $\tilde{\h}{}^{*,*}(V_r(\A^n),\Z)$ by $\alpha_i$.

The motivic cohomology of the join $V_r(\A^n)*V_s(\A^m)$ may be calculated as a module over $\m$ using the K\"unneth formula for motivic cohomology. Specifically, the natural $\A^1$-equivalence 
\[ V_r(\A^n) * V_s(\A^m) \homeq \Sigma^{1,0} (V_r(\A^n) \Smash V_s(\A^m)) \]
together with the K\"unneth formula \cite[Theorem 8.6]{Dugger2005} provides an isomorphism of $\m$-modules
\[
  \tilde{\h}{}^{*,*}(V_r(\A^n)*V_s(\A^m),\Z) \cong \left( \tilde{\h}{}^{*,*}(V_r(\A^n),\Z) \otimes_{\m} \tilde{\h}{}^{*,*}(V_s(\A^m),\Z)\right) [1],
\]
where the notation $N[1]$ denotes the $(1,0)$-shift of the bigraded module $N$. In particular, $\tilde{\h}{}^{*,*}(V_r(\A^n)*V_s(\A^m),\Z)$ is a free and finitely generated $\m$-module. If we fix presentations
\[
  \h^{*,*}(V_r(\A^n),\Z) = \m[\beta_{n-r+1}, \dots, \beta_n]/ \mathfrak{a}, \quad \h^{*,*}(V_s(\A^m),\Z) = \m[\gamma_{m-s+1}, \dots, \gamma_m]/\mathfrak{b}
\]
as in \eqref{eq:CohoPres}, then a subset of the $\m$-module generators of $\tilde{\h}{}^{*,*}(V_r(\A^n)*V_s(\A^m),\Z)$ are given by the elements $\beta_i \otimes \gamma_j[1]$ in bidegree $(2(i+j)-1,i+j)$, under the above K\"unneth isomorphism.

We need a technical lemma before we compute the map in motivic cohomology induced by $h$.
\begin{lemma}\label{lem:HoComm}
  The following hold:
  \begin{enumerate}
    \item The diagram 
    \[
      \begin{tikzcd}
        V_r(\A^n)*V_{r-1}(\A^{m-1}) \rar["p*\id"] \ar[dd,"\id*i"] & V_{r-1}(\A^n)*V_{r-1}(\A^{m-1}) \dar["h"] \\
        & V_{r-1}(\A^{n+m-1}) \dar["i"] \\
        V_r(\A^n)*V_r(\A^m) \rar["h"] & V_r(\A^{n+m})
      \end{tikzcd}
    \]
    is homotopy commutative.
    \item
    When $m$ is even, the diagram
    \[
      \begin{tikzcd}
        V_{r-1}(\A^{n-1})*V_r(\A^{m}) \rar["\id*p"] \ar[dd,"i*\id"] & V_{r-1}(\A^{n-1})*V_{r-1}(\A^{m}) \dar["h"] \\
        & V_{r-1}(\A^{n+m-1}) \dar["i"] \\
        V_r(\A^n)*V_r(\A^m) \rar["h"] & V_r(\A^{n+m})
      \end{tikzcd}
    \]
    is homotopy commutative.
  \end{enumerate}
\end{lemma}
\begin{proof}
  For the first part, it suffices to show that the two maps $V_{r|r-1}'(\A^{n|m-1}) \to V_r'(\A^{n+m})$ in the diagram of $k$-schemes 
  \begin{equation}\label{eq:HoCommModel1}
    \begin{tikzcd}
      V_{r|r-1}'(\A^{n|m-1}) \rar["p' \times \id"] \ar[dd,"\id \times i'"] & V_{r-1|r-1}'(\A^{n|m-1}) \dar[hook] \\
      & V_{r-1}'(\A^{n+m-1}) \dar["i'"] \\
      V_{r|r}'(\A^{n+m}) \rar[hook] & V_r'(\A^{n+m})
    \end{tikzcd}
  \end{equation}
  are naively $\A^1$-homotopic. On $R$-points, the lower composition sends a pair of matrices $(A,B)$ in $V_{r|r-1}'(\A^{n|m-1})(R)$ (so $A \in \Mat_{r \times n}(R)$ and $B \in \Mat_{r-1 \times m-1}(R)$) to the matrix
  \[
    \left[ 
    \begin{array}{ccc|ccc|c}
      & & & & & & 0 \\
      & \raisebox{-2ex}{$A$} & & & B & & \vdots \\
      & & & & & & 0 \\ \cline{4-7}
      & & & 0 & \cdots & 0 & 1
    \end{array}
    \right].
  \]
  If $a_{i,j}$ is the $(i,j)$-entry of $A$, then the map
  \begin{align*}
    V_{r|r-1}'(\A^{n|m-1}) \times \A^1 &\to V_r'(\A^{n+m}) \\
    ((A,B), t) &\mapsto 
    \left[ 
    \begin{array}{ccc|ccc|c}
      & & & & & & 0 \\
      & p'(A) & & & B & & \vdots \\
      & & & & & & 0 \\ \hline
      ta_{r,1}& \cdots & ta_{r,n} & 0 & \cdots & 0 & 1
    \end{array}
    \right]
  \end{align*}
  provides a naive $\A^1$-homotopy from the counterclockwise composition to the clockwise composition in \eqref{eq:HoCommModel1}.
  
  Similarly, for the second part, we wish to show that the two maps $V_{r-1|r}'(\A^{n-1|m}) \to V_r'(\A^{n+m})$ in the diagram of $k$-schemes
  \begin{equation}\label{eq:HoCommModel2}
    \begin{tikzcd}
      V_{r-1|r}'(\A^{n-1|m}) \rar["\id \times p'"] \ar[dd,"i' \times \id"] & V_{r-1|r-1}'(\A^{n-1|m}) \dar[hook] \\
      & V_{r-1}'(\A^{n+m-1}) \dar["i'"] \\
      V_{r|r}'(\A^{m|n}) \rar[hook] & V_r'(\A^{n+m})
    \end{tikzcd}
  \end{equation}
  are naively $\A^1$-homotopic. On $R$-points, the bottom composition sends a pair of matrices $(A,B)$ in $V_{r-1|r}'(\A^{n-1|m})(R)$ to the block matrix
  \[
    \left[
    \begin{array}{ccc|c|ccc}
      & & & 0 & & &\\
      & A & & \vdots & & \raisebox{-2ex}{$B$} &\\
      & & & 0 & & & \\ \cline{1-4}
      0 & \cdots & 0 & 1 & & &
    \end{array}
    \right].
  \]
  If $b_{ij}$ is the $(i,j)$-entry of $B$, then the map
  \begin{align*}
    V_{r-1|r}'(\A^{n-1|m}) \times \A^1 &\to V_r'(\A^{n+m}) \\
    ((A,B),t) &\mapsto 
    \left[
    \begin{array}{ccc|c|ccc}
      & & & 0 & & &\\
      & A & & \vdots & & p'(B) &\\
      & & & 0 & & & \\ \hline
      0 & \cdots & 0 & 1 & tb_{r,1} & \cdots & tb_{r,m}
    \end{array}
    \right]
  \end{align*}
  provides a naive $\A^1$-homotopy from the counterclockwise composition in \eqref{eq:HoCommModel2} to the map 
  \[
    g: (A,B) \mapsto 
    \left[
    \begin{array}{ccc|c|ccc}
      & & & 0 & & &\\
      & A & & \vdots & & p'(B) &\\
      & & & 0 & & & \\ \hline
      0 & \cdots & 0 & 1 & 0 & \cdots & 0
    \end{array}
    \right].
  \]
  Since $m$ is even, an even number of column transpositions provides a naive $\A^1$-homotopy from $g$ to the map
  \[
    (A,B) \mapsto 
    \left[
    \begin{array}{ccc|ccc|c}
      & & & & & & 0\\
      & A & & & p'(B) & & \vdots\\
      & & & & & & 0\\ \hline
      0 & \cdots & 0 & 0 & \cdots & 0 & 1
    \end{array}
    \right],
  \]
  which is the clockwise composition in \eqref{eq:HoCommModel2}.
\end{proof}
In what follows, let $I = \{n-r+1, \dots, n\}$ and $J = \{m-r+1, \dots, m\}$.
\begin{lemma}\label{lem:hCoho}
  Let $m$ be even. For each $\ell = n+m-r+1, \dots, n+m$, the intrinsic join 
  \[
    h: V_r(\A^n)*V_r(\A^m) \to V_r(\A^{n+m})
  \]
  induces the map in motivic cohomology 
  \begin{align*}
    h^*:\tilde{\h}{}^{2\ell-1,\ell}(V_r(\A^{n+m}),\Z) &\to \tilde{\h}{}^{2\ell-1,\ell}(V_r(\A^n)*V_r(\A^m),\Z) \\
    \alpha_{\ell} &\mapsto \sum_{\substack{(i,j) \in I \times J \\ i+j=\ell}} \pm \beta_i \otimes \gamma_j[1].
  \end{align*}
\end{lemma}
\begin{proof}
  For the case $\ell=n+m$, consider the homotopy commutative diagram
  \[
    \begin{tikzcd}
      V_r(\A^n)*V_r(\A^m) \dar["p*p"] \rar["h"] & V_r(\A^{n+m}) \dar["p"] \\
      V_1(\A^n)*V_1(\A^m) \rar["h","\weq"'] & V_1(\A^{n+m})
    \end{tikzcd}
  \]  
  whereby, after applying $\tilde{\h}{}^{2(n+m)-1,n+m}(-,\Z)$, all four motivic cohomology groups are free abelian of rank 1. Indeed, for the Stiefel varieties this follows from \cite[Theorem 19]{Williams2012}, and for the joins this follows from the discussion preceding Lemma \ref{lem:HoComm}. We chase the generator $\alpha_{n+m} \in \tilde{\h}{}^{2(n+m)-1,n+m}(V_1(\A^{n+m}),\Z)$ around the induced diagram in cohomology. The element $h^*(\alpha_{n+m})$ is a generator of 
  \[  
    \tilde{\h}{}^{2(n+m)-1,n+m}(V_1(\A^n)\otimes V_1(\A^m),\Z) = \Z \cdot \beta_n \otimes \gamma_m[1],
  \]
  since the lower intrinsic join is an $\A^1$-equivalence (Remark \ref{rem:hr=1}). We also have $p^*(\alpha_{n+m}) = \alpha_{n+m}$ under the map $p: V_r(\A^{n+m}) \to V_1(\A^{n+m})$ by \cite[Proposition 7]{Williams2012}. Similarly, $(p*p)^*(\beta_n \otimes \gamma_m[1]) = \beta_n \otimes \gamma_m[1]$. The result follows.

  When $\ell < n+m$, we prove the lemma by induction on $r$. The case $r=1$ is vacuously true. For the inductive step, let $r > 1$ and consider the homotopy commutative diagram
  \[
    \begin{tikzcd}
      V_{r-1}(\A^{n-1})*V_r(\A^m) \rar["\id*p"] \ar[dd,"i*\id"] & V_{r-1}(\A^{n-1})*V_{r-1}(\A^m) \dar["h"] \\
      & V_{r-1}(\A^{n+m-1}) \dar["i"] \\
      V_r(\A^n)*V_r(\A^m) \rar["h"] & V_r(\A^{n+m})
    \end{tikzcd}
  \]
  of Lemma \ref{lem:HoComm}. Since the group $\tilde{\h}{}^{2\ell-1,\ell}(V_r(\A^n)*V_r(\A^m),\Z)$ is free abelian, generated by the classes $\beta_i\otimes \gamma_j[1]$ such that $(i,j) \in I \times J$ and  $i+j=\ell$, we may write
  \[
    h^*(\alpha_\ell) = \sum_{\substack{(i,j) \in I \times J \\ i+j = \ell}} a_{i,j} \beta_i \otimes \gamma_j[1]
  \]
  for some integers $a_{i,j}$. Since $\ell < n+m$, we have $i^*(\alpha_\ell) = \alpha_\ell \in \tilde{\h}{}^{2\ell-1,\ell}(V_{r-1}(\A^{n+m-1}),\Z)$ by \cite[Proposition 8]{Williams2012}. The inductive hypothesis ensures that the map 
  \[
    h: V_{r-1}(\A^{n-1})*V_{r-1}(\A^m) \to V_{r-1}(\A^{n+m-1})
  \]
  induces the map in motivic cohomology that sends $\alpha_\ell$ to
  \[
    \sum_{\substack{(i,j) \in I' \times J' \\ i+j = \ell}} \pm \beta_i \otimes \gamma_j[1] \in \tilde{\h}{}^{2\ell-1,\ell}(V_{r-1}(\A^{n-1})*V_{r-1}(\A^m),\Z)
  \]
  where $I' = \{n-r+1, \dots, n-1\}$ and $J' = \{m-r+2, \dots, m\}$. Pulling back this class under $\id \otimes p$ and again using \cite[Proposition 7]{Williams2012} we obtain the class of the same name
  \[
    \sum_{\substack{(i,j) \in I' \times J' \\ i+j = \ell}} \pm \beta_i \otimes \gamma_j[1] \in \tilde{\h}{}^{2\ell-1,\ell} (V_{r-1}(\A^{n-1})*V_r(\A^m),\Z).
  \]
  We have equalities
  \[
    \sum_{\substack{(i,j) \in I' \times J' \\ i+j = \ell}} \pm \beta_i \otimes \gamma_j[1] = (i *\id)^* \Bigg( \sum_{\substack{(i,j) \in I \times J \\ i+j = \ell}} a_{i,j} \beta_i \otimes \gamma_j[1] \Bigg) = \sum_{\substack{(i,j) \in I' \times J \\ i+j=\ell}} a_{i,j} \beta_i \otimes \gamma_j[1].
  \]
  where the left equality follows from the homotopy commutativity of the diagram in question, and the right equality follows from \cite[Proposition 8]{Williams2012}. The two index sets 
  \[
    \{(i,j) \in I' \times J' \mid i+j = \ell \}, \qquad \{(i,j) \in I' \times J \mid i+j = \ell \}
  \]
  are equal since $\ell < n+m$, and we conclude that $a_{i,j} = \pm 1$ whenever $(i,j) \in I' \times J$ satisfy $i+j = \ell$. A symmetric argument using the first part of Lemma \ref{lem:HoComm} shows that $a_{n,\ell-n} = \pm 1$ as well, concluding the proof.
\end{proof}

\subsection{The stable splitting}

Given a pointed motivic space $(\sh{X},x)$, we let $M_{\red}(\sh{X})$ denote the reduced motive of $\sh{X}$: the cone of the morphism $x_*: M(\Spec k) \to M(\sh{X})$. 

\begin{proposition}\label{prop:StableSplit}
  Let $k$ be a perfect field of finite $2$-\'{e}tale cohomological dimension, and let $r \geq 2$ be an integer. Suppose a section of $p: V_r(\A^m) \to V_1(\A^m)$ exists (over $k$) for infinitely many even integers $m$. Then for any $n \geq r$, the map 
  \[
    f_r^n: \Sigma^{1,1} \tilde{\P}^{n-1}_{n-r} \to V_r(\A^n)
  \]
  has a retract in $\cat{SH}(k)$.
\end{proposition}
\begin{proof}
  Choose $m$ even so that $m \geq rn+2(r-n)$ and a section of $p: V_r(\A^m) \to V_1(\A^m)$ exists. Let $\phi:V_1(\A^m) \to V_r(\A^m)$ be such a section, and consider the solid diagram
  \begin{equation}\label{eq:Splitting}
    \begin{tikzcd}[column sep = large]
       \Sigma^{1,1} \tilde{\P}^{n-1}_{n-r} * V_1(\A^m) \ar[rr,"\tilde{\phi}",dashed,start anchor = {[yshift=.5ex]},end anchor = {[yshift=.5ex]}] \dar["f_r^n * \id"] & & \Sigma^{1,1} \tilde{\P}^{n+m-1}_{n+m-r} \dar["f_r^{n+m}"] \\
       V_r(\A^n) * V_1(\A^m) \ar[urr,dashed,"\tilde{\phi}'",start anchor = {[xshift=-1.5ex]}] \rar["\id * \phi"{yshift=-.4ex}] & V_r(\A^n) * V_r(\A^m) \rar["h"] & V_r(\A^{n+m})\: .
    \end{tikzcd}
  \end{equation}
  The Nisnevich cohomological dimension of $V_r(\A^n)*V_1(\A^m)$ is at most $rn+m$ (Remark \ref{rem:JoinCohoDim}). Since $2(n+m-r) \geq rn+m$, Lemma \ref{lem:CellApprox} and Proposition \ref{prop:frnConnectivity} guarantee that the dashed map $\tilde{\phi}'$ exists (after $\A^1$-localization), making the lower triangle in \eqref{eq:Splitting} commute. Note that there are $\A^1$-equivalences
  \[
    \Sigma^{1,1} \tilde{\P}^{n-1}_{n-r} * V_1(\A^m) \homeq \Sigma^{2m,m} \Sigma^{1,1} \tilde{\P}^{n-1}_{n-r}, \qquad V_r(\A^n) * V_1(\A^m) \homeq \Sigma^{2m,m} V_r(\A^n).
  \] 
  Let $\tilde{\phi}$ denote the composite $\tilde{\phi}' \circ (\id*f_r^n)$. Since $n+m-r-1 \ge 1$, the space $\Sigma^{1,1}\tilde{\P}^{n+m-1}_{n+m-r}$ is $\A^1$-simply connected, and we may assume $\tilde{\phi}, \tilde{\phi}'$ are pointed maps \cite[Proposition 2.1]{Gant2025}. We wish to show that $\tilde{\phi}$ is a stable $\A^1$-equivalence. 
  
  Using the conservativity theorem of \cite[Theorem 16]{Bachmann2018}, it suffices to show that the induced map of motives
  \[
    \tilde{\phi}_*: M_{\red}(\Sigma^{2m,m} \Sigma^{1,1} \tilde{\P}^{n-1}_{n-r}) \to M_{\red}(\Sigma^{1,1} \tilde{\P}^{n+m-1}_{n+m-r})
  \]
  is an isomorphism in $\cat{DM}(k)$. Since the source and target of $\tilde{\phi}_*$ are pure Tate with Tate summands concentrated on the $(2\ell-1, \ell)$-line, it suffices to show that the induced map in motivic cohomology
  \[
    \tilde{\phi}^*: \tilde{\h}{}^{2\ell-1,\ell}(\Sigma^{1,1} \tilde{\P}^{n+m-1}_{n+m-r},\Z) \to \tilde{\h}{}^{2\ell-1,\ell}(\Sigma^{2m,m}\Sigma^{1,1} \tilde{\P}^{n-1}_{n-r}, \Z)
  \]
  is an isomorphism for each integer $\ell$. 
  
  Note that the source and target of $\tilde{\phi}^*$ are either both trivial or both free abelian of rank 1. In the latter case, the element $(f_r^{n+m})^*(\alpha_{\ell})$ is a generator of $\tilde{\h}{}^{2\ell-1,\ell}(\Sigma^{1,1} \tilde{\P}^{n+m-1}_{n+m-r},\Z)$ by the calculation of \cite[Theorem 18]{Williams2012}. We also have 
  \[
    h^*(\alpha_{\ell}) = \sum_{\substack{(i,j) \in I \times J \\ i+j = \ell}} \pm \beta_i \otimes \gamma_j[1] 
  \]
  by Lemma \ref{lem:hCoho}. Since $\phi: V_1(\A^m) \to V_r(\A^m)$ is a section of $p$, we have $\phi^*(\gamma_m) = \gamma_m$. Moreover, $\phi^*(\gamma_i) = 0$ when $i<m$, as $\tilde{\h}{}^{2i-1,i}(V_1(\A^m),\Z)=0$ in this case.
  Then
  \[
    (\id * \phi)^*\Bigg( \sum_{\substack{(i,j) \in I \times J \\ i+j = \ell}} \pm \beta_i \otimes \gamma_j[1] \Bigg) = \pm \beta_{\ell-m} \otimes \gamma_{m}[1].
  \]
  Again using \cite[Theorem 18]{Williams2012}, the element $(f_r^n * \id)^*(\pm \beta_{\ell-m} \otimes \gamma_{m}[1])$ is a generator of $\tilde{\h}{}^{2\ell-1,\ell}(\Sigma^{2m,m}\Sigma^{1,1} \tilde{\P}^{n-1}_{n-r}, \Z)$. A straightforward diagram chase, noting that both triangles of \eqref{eq:Splitting} are homotopy commutative, then yields the result.
\end{proof}

The following Lemma will aid in our proof of Proposition \ref{prop:r=2Split} below in the characteristic 0 case.
\begin{lemma}\label{lem:RealUnits}
  Let $p,q$ be nonnegative integers satisfying $p - q \ge 2$, and suppose $\alpha: S^{p,q} \to S^{p,q}$ is a pointed map over $\Q$. Then $\alpha$ is an $\A^1$-equivalence if and only if the real and complex realizations of $\alpha$ are equivalences. 
\end{lemma}
\begin{proof}
  If $q=0$, there is a natural isomorphism $[S^{p,0},S^{p,0}] \iso \Z$ \cite[Corollary 6.43]{Morel2012}, and both realization maps $[S^{p,0},S^{p,0}] \to [S^p,S^p]$ are isomorphisms. If $q \ge 1$, there is a natural isomorphism $[S^{p,q},S^{p,q}] \iso \GW(\Q)$ again by \cite[Corollary 6.43]{Morel2012}. We claim that the homomorphism $i_*:\GW(\Q) \to \GW(\R)$ induced by the unique map of fields $i: \Q \to \R$ detects units; that is, $\alpha \in \GW(\Q)$ is a unit if and only if $i_*(\alpha) \in \GW(\R)$ is a unit. The Hasse-Minkowski Theorem implies that the only ring homomorphism $\W(\Q) \to \Z$ is the signature map. It follows from \cite[Prop. 2.23]{Knebusch1982} that the homomorphism $i_*: \W(\Q) \to \W(\R)$, which we may identify with the signature map, detects units. There is a presentation of $\GW(\Q)$ as a pullback
  \[
    \begin{tikzcd}
      \GW(\Q) \dar \ar[dr,phantom, very near start, "\lrcorner"] \rar["\text{rank}"] & \Z \dar \\
      \W(\Q) \rar & \Z/2
    \end{tikzcd}
  \]
  so that $\alpha$ is a unit if and only if the images of $\alpha$ under the two projections $\GW(\Q) \to \W(\Q)$ and $\GW(\Q) \to \Z$ are both units. That the two images are units may be checked by passing to $\R$; this proves the claim. With the claim in hand, the lemma follows from the fact that $i_*(\alpha) \in \GW(\R)$ is a unit if and only if its real and complex realizations have degree $\pm1$ (see, for example, \cite[pg. 18]{Wickelgren2020}).
\end{proof}

We now come to the main result of this section.

\begin{proposition}\label{prop:r=2Split}
  Let $k$ be a field. The map $f_2^n: \Sigma^{1,1} \tilde{\P}^{n-1}_{n-2} \to V_2(\A^n)$ has a retract in $\cat{SH}(k)$ which induces a splitting
  \[
    \Sigma_{\P^1}^\infty V_2(\A^n) \homeq \Sigma_{\P^1}^\infty \Sigma^{1,1}\tilde{\P}^{n-1}_{n-2} \vee S^{4n-4,2n-1}.
  \]
\end{proposition}
\begin{proof}
  It suffices to prove the proposition when $k=\Q$ and $k=\F_\ell$, where $\ell$ is an arbitrary prime. We wish to show that $f_2^n: \Sigma^{1,1} \tilde{\P}^{n-1}_{n-2} \to V_2(\A^n)$ has a stable retract; the remaining summand $S^{4n-4,2n-1}$ is identified in Remark \ref{rem:Sphere}. When $m$ is even, the map $p: V_2(\A^m) \to V_1(\A^m)$ has a section over $\Z$ and therefore over any field $k$. When $k=\F_\ell$, Proposition \ref{prop:StableSplit} applies and we are done. When $k=\Q$, we may choose $m$ even and sufficiently large to produce a map 
  \[
    \tilde{\phi'}: \Sigma^{1,1}\tilde{\P}^{n-1}_{n-2} * V_1(\A^m) \to \Sigma^{1,1} \tilde{\P}^{n+m-1}_{n+m-2}
  \]
  as in the proof of Proposition \ref{prop:StableSplit}. Recall that $\Sigma^{1,1}\tilde{\P}^{n-1}_{n-2} * V_1(\A^m) \homeq \Sigma^{2m,m} \Sigma^{1,1} \tilde{\P}^{n-1}_{n-2}$. For any integer $N \ge 2$, the motivic space $\tilde{\P}^{N-1}_{N-2}$ fits in an $\A^1$-cofibre sequence
  \begin{equation}\label{eq:eta0cofib}
    S^{2N-3,N-1} \to S^{2N-4,N-2} \to \tilde{\P}^{N-1}_{N-2}.
  \end{equation} 
  Indeed, there are $k$-rational points of $\P^{N-3}$ and $\A^{N-1} \sm \bar{0}$ fitting in a map of $\A^1$-cofibre sequences
  \[
    \begin{tikzcd}
      \Spec k \dar \rar & \P^{N-3} \dar["\imath"] \rar[equals] & \P^{N-3} \dar["\imath"] \\
      \A^{N-1} \sm \bar{0} \rar & \P^{N-2} \rar & \P^{N-1}
    \end{tikzcd}
  \]
  for which taking $\A^1$-cofibres of the vertical maps yields \eqref{eq:eta0cofib}. We remark that the first map of \eqref{eq:eta0cofib} is either $\eta$ or $0$, depending on the parity of $N$, though this point is not relevant for our arguments.
  
  Consider next the solid diagram
  \[
    \begin{tikzcd}
      S^{2n+2m-3,n+m-1} \ar[drr, dashed] \rar & \Sigma^{1,1} \tilde{\P}^{n-1}_{n-2} * V_1(\A^m) \rar \dar["\tilde{\phi}", crossing over, near start] & S^{2n+2m-1,n+m} \dar[dashed,"\alpha"] \rar & S^{2n+2m-2,n+m-1} \dar[dashed,"\beta"] \\
      S^{2n+2m-3,n+m-1} \rar & \Sigma^{1,1} \tilde{\P}^{n+m-1}_{n+m-2} \rar & S^{2n+2m-1,n+m} \rar & S^{2n+2m-2,n+m-1} 
    \end{tikzcd}
  \]
  where the rows are obtained by suspending and rotating the $\A^1$-cofibre sequence \eqref{eq:eta0cofib}. The diagonal dashed map, obtained by composing the solid maps in the diagram, is null by the connectivity calculation of \cite[Corollary 6.43]{Morel2012}. We obtain an induced map $\alpha$ as in the diagram, which further induces the map $\beta$ in the diagram. To show $\tilde{\phi}$ is a stable $\A^1$-equivalence, it suffices to show $\alpha$ and $\beta$ are $\A^1$-equivalences. Using Lemma \ref{lem:RealUnits}, we reduce to showing that the real and complex realizations of $\alpha$ and $\beta$ have degree $\pm 1$, which may be checked in singular cohomology. Indeed, it suffices to show that the real and complex realizations of $\tilde{\phi}$ induce isomorphisms in singular cohomology. This holds by a calculation in singular cohomology completely analogous to the motivic cohomology calculation of $\tilde{\phi}^*$ in the proof of Proposition \ref{prop:StableSplit}.
\end{proof}

\section{Algebraic K-theory and Adams operations}\label{sec:KT}

Let $\Gr_n(\A^m)$ denote the Grassmannian $k$-variety of $n$-planes in $m$-space. We denote by $\Gr$ the colimit of the system of inclusions
\[
  \cdots \to \Gr_n(\A^{2n}) \to \Gr_{n+1}(\A^{2(n+1)}) \to \cdots
\]
considered as a diagram in $\Spc(k)_*$. Algebraic K-theory is represented in $\cat{H}(k)_*$ by the motivic space $\Z \times \Gr$:
\[
  K_n(\sh{X}) = [S^n \Smash \sh{X}_+,\Z \times \Gr]
\]
(see \cite[Propositions 3.9 and 3.10]{Morel1999} and also \cite[Remark 2]{Schlichting2015}). If $\sh{X}$ is a pointed motivic space, we define the \emph{$n$-th reduced algebraic $K$-group} by
\[
  \tilde{K}_n(\sh{X}) := [S^n \Smash \sh{X}, \Z \times \Gr].
\]

\subsection{Adams operations}\label{subsec:AdamsOps}

Using \cite[Theorem 0.2]{Riou2010}, the Adams operations $\psi^k$ on $K_0(-)$ define endomorphisms of $\Z \times \Gr$ in $\cat{H}(k)_*$, which we also denote by $\psi^k$. These endomorphisms induce Adams operations on $K_n(-)$ and $\tilde{K}_n(-)$ via the above identifications. It is not clear that the Adams operations $\psi^k$ on $K_n(-)$ for $n \geq 1$ constructed this way agree with the Adams operations on higher $K$-groups constructed in \cite{Grayson1992}, but we do not pursue this point since it is not relevant to the arguments of this paper.  

Let $\xi_\infty \in K_0(\Z \times \Gr)$ denote the unique element corresponding to the identity map $\Z \times \Gr \to \Z \times \Gr$. We let $\eta$ denote the class $[\sh{O}(-1)]-1 \in K_0(\P^1)$. Voevodsky's $\P^1$-spectrum representing algebraic $K$-theory \cite[Section 6.2]{Voevodsky1998} has bonding maps
\[
  \sigma: \P^1 \Smash (\Z \times \Gr) \to \Z \times \Gr
\]
with the following properties (see \cite[pg. 2]{Panin2009}):
\begin{enumerate}
    \item the morphism
    \[
      \P^1 \times \Z \times \Gr \to \P^1 \Smash (\Z \times \Gr) \xrar{\sigma} \Z \times \Gr,
    \]
    where the first map is the canonical one, represents the element $\eta \otimes \xi_\infty \in K_0(\P^1 \times \Z \times \Gr)$;
    \item the adjoint map $\bar{\sigma}: \Z \times \Gr \to \Omega_{\P^1} (\Z \times \Gr)$ is an $\A^1$-equivalence. 
\end{enumerate}

\begin{lemma}\label{lem:AdamsOps}
  Let $\sh{X}$ be a pointed motivic space. There is an isomorphism $\phi_\sh{X}: \tilde{K}_1(\Sigma^{1,1} \sh{X}) \to \tilde{K}_0(\sh{X})$, natural in $\sh{X}$, such that 
  \[
    \begin{tikzcd}
      \tilde{K}_1(\Sigma^{1,1} \sh{X}) \rar["\phi_\sh{X}","\iso"'] \dar["\psi^k"] & \tilde{K}_0(\sh{X}) \dar["k\psi^k"] \\
      \tilde{K}_1(\Sigma^{1,1} \sh{X}) \rar["\phi_\sh{X}","\iso"'] & \tilde{K}_0(\sh{X})
    \end{tikzcd}
  \]
  commutes.
\end{lemma}
\begin{proof}
  The argument is adapted from \cite[pg. 99]{Adams1974}. First, we claim that the right square in
  \begin{equation}\label{eq:Adj}
    \begin{tikzcd}
      \P^1 \times \Z \times \Gr \rar \dar["\id \times k\psi^k"] & \P^1 \Smash (\Z \times \Gr) \rar["\sigma"] \dar["\id \Smash k\psi^k"] & \Z \times \Gr \dar["\psi^k"] \\
      \P^1 \times \Z \times \Gr \rar & \P^1 \Smash (\Z \times \Gr) \rar["\sigma"] & \Z \times \Gr
    \end{tikzcd}
  \end{equation}
  is commutative, where the leftmost horizontal maps are the canonical ones. It suffices to show that the outer square commutes, considered as a diagram in $\cat{H}(k)$. The clockwise composition corresponds to the element 
  \[
    \psi^k(\eta \otimes \xi_\infty) = \psi^k(\eta) \otimes \psi^k(\xi_\infty) = k\eta \otimes \psi^k(\xi_\infty) = \eta \otimes k\psi^k(\xi_\infty)
  \]
  of $K_0(\P^1 \times \Z \times \Gr)$, which proves the claim. The adjoint of the right square in \eqref{eq:Adj} is the commuting square
  \begin{equation}\label{eq:Adj2}
    \begin{tikzcd}
      \Z \times \Gr \rar["\weq"',"\bar{\sigma}"] \dar["k\psi^k"] & \Omega_{\P^1} (\Z \times \Gr) \dar["\Omega_{\P^1}\psi^k"] \\
      \Z \times \Gr \rar["\weq"',"\bar{\sigma}"] & \Omega_{\P^1} (\Z \times \Gr).
    \end{tikzcd}
  \end{equation}
  Consider the chain of isomorphisms
  \[
    \tilde{K}_1(\Sigma^{1,1} \sh{X}) = [S^1 \Smash \Sigma^{1,1} \sh{X}, \Z \times \Gr] \iso [\sh{X}, \Omega_{\P^1} (\Z \times \Gr)] \iso [\sh{X}, \Z \times \Gr] = \tilde{K}_0(\sh{X})
  \]
  of abelian groups. The first isomorphism $[S^1 \Smash \Sigma^{1,1} \sh{X}, \Z \times \Gr] \iso [\sh{X}, \Omega_{\P^1} (\Z \times \Gr)]$ respects the endomorphisms
  \[
    \psi^k : \Z \times \Gr \to \Z \times \Gr, \qquad \Omega_{\P^1}\psi^k: \Omega_{\P^1} (\Z \times \Gr) \to \Omega_{\P^1} (\Z \times \Gr)
  \]
  by adjunction. The second isomorphism $[\sh{X}, \Omega_{\P^1} (\Z \times \Gr)] \iso [\sh{X}, \Z \times \Gr]$ respects the endomorphisms
  \[
    \Omega_{\P^1}\psi^k: \Omega_{\P^1} (\Z \times \Gr) \to \Omega_{\P^1} (\Z \times \Gr), \qquad k\psi^k: \Z \times \Gr \to \Z \times \Gr 
  \]
  by the commutativity of \eqref{eq:Adj2}.
\end{proof}

\subsection{\texorpdfstring{$K$-theoretic obstructions}{K-Theoretic obstructions}}

The $K$-theory of projective space is given by
\[
  K_0(\P^{n-1}) = \frac{\Z[\mu]}{(\mu^n)}
\]
where $\mu = [\mathcal{O}(1)]-1$ (see, for instance, \cite[Theorem 4.5]{Manin1969}). We identify the reduced $K$-group $\tilde{K}_0(\P^{n-1})$ with the free abelian subgroup of $K_0(\P^{n-1})$ generated by the elements $\mu, \mu^2, \dots, \mu^{n-1}$. Moreover, from the long exact sequence in $K$-theory associated with the $\A^1$-cofibre sequence
\[
  \P^{m-1} \xrar{\imath} \P^{n-1} \xrar{\rho} \P^{n-1}_m,
\]
and the fact that 
\[
  \imath^*(\mu^i) = 
  \begin{cases} 
    \mu^i & \text{if } 1 \le i \le m-1 \\
    0 & \text{if } m \le i \le n-1
  \end{cases},
\]
the group $\tilde{K}_0(\P^{n-1}_{m})$ can be identified with the free abelian subgroup of $\tilde{K}_0(\P^{n-1})$ generated by $\mu^m, \dots, \mu^{n-1}$. We will need to describe the action of $\psi^2$ on $\tilde{K}_0(\P^{n-1}_m)$, which can be deduced from the action of $\psi^2$ on $\tilde{K}_0(\P^{n-1})$. From $\psi^2(\mu) = \mu^2+2\mu$, and since $\psi^2$ is a ring homomorphism, in $\tilde{K}_0(\P^{n-1})$ we have
\begin{equation}\label{eq:psi2}
  \psi^2(\mu^i) = (\mu^2+2\mu)^i = \sum_{j=0}^i \binom{i}{j} 2^{i-j} \mu^{i+j} \mod{(\mu^n)}
\end{equation}
for each $i=1, \dots, n-1$. Since $\rho^*(\mu^i) = \mu^i$, the equality \eqref{eq:psi2} holds in $\tilde{K}_0(\P^{n-1}_{m})$ when $i = m , \dots, n-1$.

\begin{lemma}\label{lem:NoRetract}
  Let $n \geq 5$ be an integer. There does not exist a retract of 
  \[
    \Sigma^{1,1}\rho^*: \tilde{K}_1(\Sigma^{1,1} \P^{n-1}_{n-2}) \to \tilde{K}_1(\Sigma^{1,1} \P^{n-1}_{n-4})
  \]
  that commutes with Adams operations (as defined in \Cref{subsec:AdamsOps}).
\end{lemma}
\begin{proof}
  Suppose such a retract $\phi': \tilde{K}_1(\Sigma^{1,1} \P^{n-1}_{n-4}) \to \tilde{K}_1(\Sigma^{1,1} \P^{n-1}_{n-2})$ exists for the sake of contradiction. Using Lemma \ref{lem:AdamsOps}, there exists a retract of 
  \[
    \rho^*: \tilde{K}_0(\P^{n-1}_{n-2}) \to \tilde{K}_0(\Sigma^{1,1} \P^{n-1}_{n-4})
  \]
  which commutes with the operations $k\psi^k$. Call this retract $\phi$. Since $\tilde{K}_0(\P^{n-1}_{n-2})$ is a free abelian group, the retract $\phi$ also commutes with the Adams operations $\psi^k$. 
  
  The remainder of the proof proceeds as in \cite[Satz 1]{Suter1970}; we include the argument for completeness. Since $\phi$ is a retract of $\rho^*$ and $\rho^*(\mu^i) = \mu^i$ for $i=n-2,n-1$, we have $\phi(\mu^i) = \mu^i$ for $i=n-2,n-1$. Write
  \[
    \phi(\mu^{n-4}) = a\mu^{n-2} + b\mu^{n-1}, \qquad \phi(\mu^{n-3}) = c\mu^{n-2}+d\mu^{n-1}
  \]
  for some integers $a,b,c,d$. Using \eqref{eq:psi2}, we have
  \begin{align*}
    \phi \circ \psi^2(\mu^{n-4}) &= (2^{n-4}a+(n-4)2^{n-5}c+(n-4)(n-5)2^{n-7})\mu^{n-2}\\ 
    &\quad+ \left( 2^{n-4}b+(n-4)2^{n-5}d+\frac{(n-4)(n-5)(n-6)}{3}2^{n-8} \right) \mu^{n-1}, \\
    \psi^2 \circ \phi(\mu^{n-4}) &= (2^{n-2}a)\mu^{n-2}+((n-2)2^{n-3}a+2^{n-1}b)\mu^{n-1},\\
    \phi \circ \psi^2(\mu^{n-3}) &= (2^{n-3}c+(n-3)2^{n-4})\mu^{n-2}+(2^{n-3}d+(n-3)(n-4)2^{n-6})\mu^{n-1},\\
    \psi^2 \circ \phi(\mu^{n-3}) &= (2^{n-2}c)\mu^{n-2}+((n-2)2^{n-3}c+2^{n-1}d)\mu^{n-1}.
  \end{align*}
  Comparing coefficients, we find
  \begin{gather*}
    2^{n-2}c = 2^{n-3}c+(n-3)2^{n-4}, \qquad 2^{n-2}a = 2^{n-4}a+(n-4)2^{n-5}c+(n-4)(n-5)2^{n-7}, \\
    (n-2)2^{n-3}c+2^{n-1}d = 2^{n-3}d+(n-3)(n-4)2^{n-6}.
  \end{gather*}
  From the first equation, we get $c=(n-3)/2$ so that $n$ is odd. From the second and third equations, we obtain
  \[
    a = \frac{3n^2-23n+44}{24}, \qquad d = \frac{-3n^2+13n-12}{24}.
  \]
  Then
  \[
    d = \frac{-3n^2+23n-44}{24}+ \frac{-10n+32}{24} = -a+ \frac{-5n+16}{12}.
  \]
  Since $n$ is odd, we observe $d \notin \Z$, a contradiction. 
\end{proof}

\begin{lemma}\label{lem:NoRetract2}
  Let $n \ge 4$ be an integer. If there exists a retract of 
  \[
    \Sigma^{1,1}\rho^*: \tilde{K}_1(\Sigma^{1,1} \P^{n-1}_{n-2}) \to \tilde{K}_1(\Sigma^{1,1}\P^{n-1}_{n-3})
  \]
  that commutes with Adams operations (as defined in \Cref{subsec:AdamsOps}), then $n \equiv 3 \pmod{24}$
\end{lemma}
\begin{proof}
  The proof is more or less contained in the proof of Lemma \ref{lem:NoRetract}. There is a retract $\phi$ of 
  \[
    \rho^*: \tilde{K}_0(\P^{n-1}_{n-2}) \to \tilde{K}_0(\P^{n-1}_{n-3})
  \]
  which commutes with $\psi^k$. Again write $\phi(\mu^{n-3}) = c\mu^{n-2}+d\mu^{n-1}$ for some integers $c,d$. As in the proof of Lemma \ref{lem:NoRetract}, we have
  \begin{align*}
    \phi \circ \psi^2(\mu^{n-3}) =& (2^{n-3}c+(n-3)2^{n-4})\mu^{n-2}+(2^{n-3}d+(n-3)(n-4)2^{n-6})\mu^{n-1},\\
    \psi^2 \circ \phi(\mu^{n-3}) =& (2^{n-2}c)\mu^{n-2}+((n-2)2^{n-3}c+2^{n-1}d)\mu^{n-1},
  \end{align*}
  so that
  \[
    2^{n-2}c = 2^{n-3}c+(n-3)2^{n-4}, \qquad (n-2)2^{n-3}c+2^{n-1}d = 2^{n-3}d+(n-3)(n-4)2^{n-6}.
  \]
  Then $c=(n-3)/2$, so $n$ is odd, and 
  \[
    d = \frac{-3n^2+13n-12}{24} \in \Z.
  \]
  The integer $-3n^2+13n-12$ is divisible by $24$ if and only if $n \equiv 3 \pmod{24}$ or $n \equiv 12 \pmod{24}$. Since $n$ is odd, we conclude. 
\end{proof}

\section{The nonexistence of sections}\label{sec:Sections}

We may now prove the main results of this paper concerning the nonexistence of a section of $p: V_{r+\ell}(\A^n) \to 
V_r(\A^n)$.

\subsection{The case of \texorpdfstring{$\ell \ge 2$}{l>=2}}

\begin{theorem}\label{thm:l>1}
  Let $k$ be a field and $r,\ell,n$ be integers with $r,\ell \geq 2$ and $r + \ell \leq n$. The projection $p\co V_{r+\ell}(\A^n) \to V_r(\A^n)$ does not have a section over $k$. 
\end{theorem}
\begin{proof}
  Using Lemma \ref{lem:Sec1}, we may reduce to the case $\ell=2$, and using Lemma \ref{lem:Sec2}, we may reduce to the case $r=2$. That is, we wish to show that there does not exist a section of $p: V_4(\A^n) \to V_2(\A^n)$. For the sake of contradiction, suppose a section $\phi: V_2(\A^n) \to V_4(\A^n)$ of $p$ exists.
  
  \textbf{Case 1: $n \ge 8$}. We consider the solid diagram
  \begin{equation}\label{eq:lift1}
    \begin{tikzcd}
      \Sigma^{1,1} \tilde{\P}^{n-1}_{n-2} \dar["f_2^n"] \rar[dashed, "\tilde{\phi}"] & \Sigma^{1,1} \tilde{\P}^{n-1}_{n-4} \dar["f_4^n"] \\
      V_2(\A^n) \rar["\phi"] & V_4(\A^n)
    \end{tikzcd}.
  \end{equation}
  The Nisnevich cohomological dimension of $\Sigma^{1,1} \tilde{\P}^{n-1}_{n-2}$ is $n$, and the map $f_4^n$ is $\A^1$-$(2n-9)$-connected by Proposition \ref{prop:frnConnectivity}. As $n \ge 8$, Lemma \ref{lem:CellApprox} guarantees that the dashed lift $\tilde{\phi}: \Sigma^{1,1}\tilde{\P}^{n-1}_{n-2} \to \Sigma^{1,1}\tilde{\P}^{n-1}_{n-4}$ of $\phi \circ f_2^n$ exists making \eqref{eq:lift1} commute (after $\A^1$-localization). The space $\Sigma^{1,1}\tilde{\P}^{n-1}_{n-4}$ is $\A^1$-simply connected in this case (Lemmas \ref{lem:iConn2} and \ref{lem:SmashConn}), so, using \cite[Proposition 2.1]{Gant2025}, we may assume $\tilde{\phi}$ is a pointed map. In particular, the map on $\tilde{K}_1(-)$ induced by $\tilde{\phi}$ commutes with the Adams operations. 
  
  The stable splitting of Proposition \ref{prop:r=2Split} implies that the induced map 
  \[
    {f_2^n}^*: \tilde{K}_1(V_2(\A^n)) \to \tilde{K}_1(\Sigma^{1,1}\tilde{\P}^{n-1}_{n-2})
  \]
  is an epimorphism. Using the commutativity of \eqref{eq:prho}, we also have
  \[
    \tilde{\phi}^* \circ \Sigma^{1,1}\tilde{\rho}^* \circ {f_2^n}^* = \tilde{\phi}^* \circ {f_4^n}^* \circ p^* = {f_2^n}^* \circ \phi^* \circ p^* = {f_2^n}^*
  \]
  on $\tilde{K}_1(-)$, so that $\tilde{\phi}^* \circ \Sigma^{1,1}\tilde{\rho}^* = \id$. This contradicts Lemma \ref{lem:NoRetract}.

  \textbf{Case 2: $n < 8$}. If a section of $p: V_4(\A^n) \to V_2(\A^n)$ exists, then there exists a section of $p\co V_3(\A^{n-1}) \to V_1(\A^{n-1})$ as well by Lemma \ref{lem:Sec2}. 
  
  If the characteristic of $k$ is $0$, then the existence of a section of $p: V_3(\A^{n-1}) \to V_1(\A^{n-1})$ implies that $n-1$ is divisible by the third James number $b_3=24$ by \cite[Th\'eor\`eme 6.5]{Raynaud1968}. We conclude that no such section exists when $n < 8$ in the characteristic-$0$ case.
  
  If the characteristic of $k$ is $p>0$, then \cite[Th\'eor\`eme 6.6]{Raynaud1968} implies that $n-1$ is divisible by a certain integer $N_3(p)$ (see \cite[pg. 21]{Raynaud1968} for a definition of this integer). 
  
  When $p = 2$, the integer $N_3(2)$ is $3$, so a section of $p: V_3(\A^{n-1}) \to V_1(\A^{n-1})$ exists possibly in the cases $n-1=3,6$. When $n-1=3$, the example of M. Kumar and M.V. Nori in \cite[pg. 1443]{MohanKumar1985} provides a stably free module of rank $2$ over a $k$-algebra given by a unimodular row of length $3$ that is not free, so the map $p: \GL_3 \to V_1(\A^3)$ does not have a section over $k$ (see \cite[Proposition 2.4]{Raynaud1968}). In the case $n-1=6$, we consider the action of the Steenrod squares in characteristic $2$ of \cite{Primozic2020}. For $i=4,5,6$, let $\bar{\alpha}_i$ be the mod-2 reduction of the class $\alpha_i \in \h^{2i-1,i}(V_3(\A^6),\Z)$ (see \eqref{eq:CohoPres}). If a section of $p: V_3(\A^6) \to V_1(\A^6)$ exists, then the Steenrod squares must vanish on $\bar{\alpha}_4$. The calculation $\text{Sq}^4(\bar{\alpha}_4) = \bar{\alpha}_6$ in $\h^{*,*}(V_3(\A^6),\Z/2)$ of \cite[Proposition 2.3]{Primozic2022} obstructs the existence of a section of $p: V_3(\A^6) \to V_1(\A^6)$. 

  In the case of $p=3$, the integer $N_3(3)$ is $4$, so we need to show that there does not exist a section of $p: V_3(\A^4) \to V_1(\A^4)$. The calculation $\text{Sq}^4(\bar{\alpha}_2) = \bar{\alpha}_4$ in $\h^{*,*}(V_3(\A^4),\Z/2)$ of \cite[Theorem 20]{Williams2012} obstructs the existence of such a section in this case. 

  Finally, if $p > 3$, the integer $N_3(p)$ is $12$, so there does not exist a section of $p: V_3(\A^{n-1}) \to V_1(\A^{n-1})$ when $n < 8$.
\end{proof}

\begin{remark}
  Using \cite[Proposition 4.6]{Gant2025}, Theorem \ref{thm:l>1} shows that there does not exist a section of $p$ in $\cat{H}(k)$. In particular, there does not exist a section of $p'\co V_{r+\ell}'(\A^n) \to V_r'(\A^n)$ over $k$ when $r,\ell \ge 2$.
\end{remark}

\subsection{The case of \texorpdfstring{$\ell=1$}{l=1}}

\begin{theorem}\label{thm:l=1}
  Let $k$ be a field and $r,n \ge 2$ be integers with $r \le n-2$. If the projection $p: V_{r+1}(\A^n) \to V_r(\A^n)$ has a section over $k$, then $n-r \equiv 1 \pmod{24}$.
\end{theorem}
\begin{proof}
    Since $p: V_{r+1}(\A^n) \to V_r(\A^n)$ has a section, there exists a section $\phi$ of the map $p: V_3(\A^{n-r+2}) \to V_2(\A^{n-r+2})$ by Lemma \ref{lem:Sec2}. 
    
    \textbf{Case 1: $n-r \geq 3$}. We consider the solid diagram
    \begin{equation}\label{eq:lift2}
      \begin{tikzcd}
        \Sigma^{1,1} \tilde{\P}^{n-r+1}_{n-r} \rar[dashed, "\tilde{\phi}"] \dar{f_2^{n-r+2}} & \Sigma^{1,1}\tilde{\P}^{n-r+1}_{n-r-1} \dar{f_3^{n-r+2}} \\
        V_2(\A^{n-r+2}) \rar{\phi} & V_3(\A^{n-r+2})
      \end{tikzcd}.
    \end{equation}
    Note that $2n-2r-2 \ge n-r+2$ in this case. The map $f_3^{n-r+2}$ is $\A^1$-$2n-2r-3$-connected (Proposition \ref{prop:frnConnectivity}), while the Nisnevich cohomological dimension of $\Sigma^{1,1} \tilde{\P}^{n-r+1}_{n-r-1}$ is $n-r+2$. Lemma \ref{lem:CellApprox} implies that the dashed map $\tilde{\phi}: \Sigma^{1,1} \tilde{\P}^{n-r+1}_{n-r} \to \Sigma^{1,1}\tilde{\P}^{n-r+1}_{n-r-1}$ exists making the square of \eqref{eq:lift2} commute (after $\A^1$-localization). The space $\Sigma^{1,1}\tilde{\P}^{n-r+1}_{n-r-1}$ is $\A^1$-simply connected in this case by Lemma \ref{lem:iConn2}, so we may assume $\tilde{\phi}$ is a pointed map by \cite[Proposition 2.1]{Gant2025}. In particular, $\tilde{\phi}$ commutes with Adams operations on $\tilde{K}_1(-)$.
    
    The induced map 
    \[
      {f_2^{n-r+2}}^*: \tilde{K}_1(V_2(\A^{n-r+2})) \to \tilde{K}_1(\Sigma^{1,1}\tilde{\P}^{n-r+1}_{n-r})
    \]
    is an epimorphism by Proposition \ref{prop:r=2Split} which implies that the map $\tilde{\phi}^*$ is a retract of 
    \[
      \Sigma^{1,1} \tilde{\rho}^*: \tilde{K}_1(\Sigma^{1,1}\tilde{\P}^{n-r-1}_{n-r}) \to \tilde{K}_1(\Sigma^{1,1}\tilde{\P}^{n-r+1}_{n-r-1}).
    \]
    We conclude from Lemma \ref{lem:NoRetract2} that $n-r+2 \equiv 3 \pmod{24}$, or $n-r \equiv 1 \pmod{24}$.

    \textbf{Case 2: $n-r=2$}. We wish to show a section of $p: V_3(\A^4)\to V_2(\A^4)$ does not exist. If there were such a section, then $p: V_2(\A^3) \to V_1(\A^3)$ has a section as well by Lemma \ref{lem:Sec2}. If $\bar{\alpha}_2,\bar{\alpha}_3$ denote the mod-2 reductions of the $\m$-algebra generators of $\h^{*,*}(V_2(\A^3),\Z)$, then the calculation $\text{Sq}^2(\bar{\alpha}_2) = \bar{\alpha}_3$ of \cite[Theorem 20]{Williams2012} (or \cite[Proposition 2.3]{Primozic2022} for the case of $\char(k)=2$) shows that a section of $p: V_2(\A^3) \to V_1(\A^3)$ does not exist. 
\end{proof}

\subsection{Examples of stably free modules without free summands}\label{sec:Modules}

Recall from the introduction that \cite[Proposition 2.4]{Raynaud1968} asserts that the map $V_{r+\ell}(\A^n) \to V_r(\A^n)$ has a section over $k$ if and only if the universal stably free module $P_{n,n-r}$ has a free summand of rank $\ell$. Theorems \ref{thm:l>1} and \ref{thm:l=1} thus have the following interpretation.

\begin{theorem}\label{thm:Main}
  Let $k$ be a field, and $r,n$ be positive integers satisfying $2 \le r \le n-2$. There is a $k$-algebra $R$ and a stably free $R$-module $P$ of type $(n,n-r)$ that does not admit a free summand of rank 2. Suppose further that $n-r \not\equiv 1 \pmod{24}$, then $P$ may be chosen so as not to admit a free summand of rank 1.
\end{theorem}

\appendix

\section{Connectivity and essentially smooth base change}\label{sec:BaseChange}

Let $E/k$ be a field extension. Given a Nisnevich sheaf $\sh{X}$ over $k$, we denote by $\sh{X}_E$ the base change of $\sh{X}$ along $E/k$. We set out to prove the following lemma, which shows that connectivity interacts well with base change along $E/k$, at least when $k$ is perfect. This result is probably known to experts (see the appendix of \cite{Hoyois2015} and \cite[Paragraph 2.1.5]{Asok2024}), but does not seem to exist in the literature. 

\begin{lemma}\label{lem:BaseChange}
  Let $k$ be a perfect field, and suppose $f:\sh{X} \to \sh{Y}$ is a map in $\Shv_{\Nis}(\Sm_k)$. If $E/k$ is an extension and $n \ge -2$ is an integer, the following hold:
  \begin{enumerate}
    \item\label{i} If $\sh{X}$ is $\A^1$-$n$-connected, then $\sh{X}_E$ is $\A^1$-$n$-connected.
    \item\label{ii} If $f$ is $\A^1$-$n$-connected, then $f_E: \sh{X}_E \to \sh{Y}_E$ is $\A^1$-$n$-connected. 
  \end{enumerate}
\end{lemma}

\subsection{Points and connectivity}\label{sec:Points}
Let $T$ be a Noetherian scheme of finite Krull dimension, and let $\EuScript{S}$ denote the $\infty$-category of spaces. Following \cite[Remark 6.5.4.7]{Lurie2009}, a \emph{point} $p=(p^*,p_*)$ of $\Shv_{\Nis}(\Sm_T)$ is an adjunction 
\[
  p^* : \Shv_{\Nis}(\Sm_T) \rightleftarrows \EuScript{S} : p_*
\]
where the left adjoint $p^*$ preserves finite limits. 

A pair $(U,u)$, where $U$ is an object of $\Sm_T$ and $u \in U$, induces a point $p=(p^*,p_*)$ where $p^*$ is the Nisnevich stalk functor at $u \in U$. That is, 
for an object $\sh{X}$ of $\Shv_{\Nis}(\Sm_T)$, 
\[
  p^*\sh{X} = \colim_{V \in \sh{N}_u}\sh{X}(V)
\]
where $\sh{N}_u$ is the cofiltered diagram of Nisnevich open neighborhoods of $u \in U$. Points of this form make up a \emph{conservative family of points}. That is, a morphism $f: \sh{X} \to \sh{Y}$ in $\Shv_{\Nis}(\Sm_T)$ is an equivalence if and only if $p^*f$ is an equivalence for all points $p$ induced by pairs $(U,u)$ (see \cite[Lemma 3.1.11]{Morel1999}). We refer to the family of points induced by pairs $(U,u)$, where $U$ varies over all objects of $\Sm_T$ and $u$ varies over all points of $U$, as \emph{the conservative family over $T$}. A consequence of conservativity is that a morphism $f:\sh{X} \to \sh{Y}$ in $\Shv_{\Nis}(\Sm_T)$ is $n$-connected if and only if $p^*f: p^*\sh{X} \to p^*\sh{Y}$ is an $n$-connected map of spaces (i.e., has $n$-connected homotopy fibres) for all points $p=(p^*,p_*)$ in the conservative family over $T$.  

\subsection{Technical lemmas}
We summarize results from the appendix of \cite{Hoyois2015} that we will need. 

\begin{definition}\label{def:EssSmooth}
  A morphism of schemes $T \to T'$ is \emph{essentially smooth} if $T$ may be written as a cofiltered limit $\lim_\alpha T_\alpha$ where $T_\alpha$ is smooth over $T'$ and the transition morphisms $T_\alpha \to T_\beta$ are affine.
\end{definition}

As in the above situation, suppose $j: \Spec E \to \Spec k$ is a morphism of schemes corresponding to an extension $E/k$ with $k$ perfect. The following lemmas, all due to M. Hoyois, can be found in \cite{Hoyois2015}.

\begin{lemma}[{\cite[Lemma A.2]{Hoyois2015}}]\label{lem:1}
  The morphism $j: \Spec E \to \Spec k$ is essentially smooth. 
\end{lemma}

\begin{lemma}[{\cite[Lemma A.4]{Hoyois2015}}]\label{lem:2}
  The base change $j^*:\Pre(\Sm_k) \to \Pre(\Sm_E)$ preserves $\A^1$-local and Nisnevich-local objects. 
\end{lemma}
For the remainder of this appendix, we fix a presentation $\Spec E = \lim_{\alpha} T_\alpha$ as in Definition \ref{def:EssSmooth} (which exists by Lemma \ref{lem:1}). We remark that each $T_\alpha$ is affine but not necessarily a field. Let $f_\alpha: \sh{X}_\alpha \to \sh{Y}_\alpha$ denote the base change of $f: \sh{X} \to \sh{Y}$ to $T_\alpha$, and let $\Map_\alpha(-,-)$ denote the space of maps in $\Shv_{\Nis}(\Sm_{T_\alpha})$. 
\begin{lemma}[{\cite[Lemma A.5]{Hoyois2015}}]\label{lem:3}
  If $D:I \to \Sm_E$ is a finite diagram which is a cofiltered limit of diagrams $D_\alpha: I \to \Sm_{T_\alpha}$, and $\sh{X}$ is an object of $\Shv_{\Nis}(\Sm_k)$, then the natural map
  \[
    \colim_\alpha \Map_\alpha(\colim D_\alpha , \sh{X}_{\alpha}) \to \Map_E(\colim D, \sh{X}_E)
  \]
  is an equivalence.
\end{lemma}
We now turn to the proof of Lemma \ref{lem:BaseChange}.
\begin{proof}[Proof of Lemma \ref{lem:BaseChange}]
  By considering the map $\sh{X} \to *$, it suffices to prove \eqref{ii} of Lemma \ref{lem:BaseChange}. Also, using Lemma \ref{lem:2}, it suffices to show that if $f:\sh{X} \to \sh{Y}$ is $n$-connected, then $f_E:\sh{X}_E \to \sh{Y}_E$ is $n$-connected. The strategy is first to prove that $f_\alpha: \sh{X}_\alpha \to \sh{Y}_\alpha$ is $n$-connected, and then use Lemma \ref{lem:3} to prove that $f_E$ is $n$-connected. 

  \textbf{The proof that $f_\alpha$ is $n$-connected.} It suffices to show that $p^*f_\alpha$ is $n$-connected for all points $p$ in the conservative family over $T_\alpha$. Any point of $\Shv_{\Nis}(\Sm_{T_\alpha})$ induced by $(U_\alpha,u_\alpha)$, where $U_\alpha$ is smooth over $T_\alpha$ and $u_\alpha \in U_\alpha$, is also a point in the conservative family over $k$, as $U_\alpha$ is smooth over $k$. We conclude.

  \textbf{The proof that $f_E$ is $n$-connected.} 
  Let $p$ be a point in the conservative family over $E$ induced by the pair $(U,u)$. We may write $U = \lim_\alpha U_\alpha$ where $U_\alpha$ are objects of $\Sm_{T_\alpha}$ (see \cite[Appendix A]{Hoyois2015}). For each $\alpha$, the element $u \in U$ determines an element $u_\alpha \in U_\alpha$ via the structure morphism $U \to U_\alpha$. 
  
  Using \cite[Th\'{e}or\`{e}me 8.8.2 (2)]{Grothendieck1966}, each $V \to U$ in $\sh{N}_u$ is the limit of a diagram $(V_\alpha)$ over $(U_\alpha)$, and we may suppose each $V_\alpha$ is an object of $\Sm_{T_\alpha}$. Using \cite[Proposition 17.7.8]{Grothendieck1967}, we may further suppose that each $V_\alpha \to U_\alpha$ is a Nisnevich open neighbourhood of $u_\alpha$. In other words, the diagram $\sh{N}_u$ is the cofiltered limit of the diagrams $\sh{N}_{u_\alpha}$. It follows that there is an equivalence, natural in $\sh{X}$, 
  \[
    \colim_\alpha \colim_{V_\alpha \in \sh{N}_{u_\alpha}} \Map_\alpha (V_\alpha,\sh{X}_\alpha) \homeq \colim_{V \in \sh{N}_u} \colim_\alpha \Map_\alpha (V_\alpha, \sh{X}_\alpha).
  \]
  Composing with the equivalence of Lemma \ref{lem:3}, we obtain a natural equivalence
  \[
    \colim_\alpha \colim_{V_\alpha \in \sh{N}_{u_\alpha}} \Map_\alpha (V_\alpha,\sh{X}_\alpha) \to \colim_{V \in \sh{N}_u} \Map_E(V,\sh{X}_E).
  \]
  If we denote by $p_\alpha=(p_\alpha^*,{p_\alpha}_*)$ the point of $\Shv_{\Nis}(\Sm_{T_\alpha})$ corresponding to $(U_\alpha,u_\alpha)$, then there are natural equivalences
  \[
    \colim_{V \in \sh{N}_u} \Map_E(V,\sh{X}_E) \homeq p^* \sh{X}_E, \qquad \colim_\alpha \colim_{V_\alpha \in \sh{N}_{u_\alpha}} \Map_\alpha (V_\alpha,\sh{X}_\alpha) \homeq \colim_\alpha p_\alpha^* \sh{X}_\alpha.
  \]
  Hence, we have a diagram 
  \[
    \begin{tikzcd}[column sep = 6em]
      \colim_\alpha p_\alpha^*\sh{X}_\alpha \rar{\colim_\alpha p_\alpha^*(f_\alpha)} \dar & \colim_\alpha p_\alpha^*\sh{Y}_\alpha \dar \\
      p^*\sh{X}_E \rar{p^*f_E} & p^* \sh{Y}_E
    \end{tikzcd}
  \]
  where the vertical maps may be identified with the equivalences of Lemma \ref{lem:3}. Since each $f_\alpha$ is $n$-connected and any filtered colimit of $n$-connected maps of spaces is $n$-connected, we are done. 
\end{proof}

\section{\texorpdfstring{The proof of Lemma \ref{lem:IncBasepoint}}{The proof of Lemma 5.6}}\label{sec:ProofOfInc}

We will call a closed immersion of smooth $k$-schemes $Z \hookrightarrow X$ a \emph{smooth pair}, which we write as $(X,Z)$. We denote the normal bundle of $Z$ in $X$ by $N_Z X$ and its associated Thom space $\Th(N_Z X)$. Following \cite{Hoyois2017} and \cite{Asok2023}, a \emph{map of smooth pairs} $f: (X,Z) \to (X',Z')$ is a morphism of $k$-schemes $f: X \to X'$ such that $f$ restricts to a morphism of $k$-schemes $Z \to Z'$. We say that $f$ is \emph{transversal}, or \emph{$f$ is transverse to $Z'$}, if $f^{-1}(Z') = Z$ and the induced map of vector bundles
\[ N_Z X \to f^* N_{Z'} X' \]
is an isomorphism. 

We use the formulation of homotopy purity in \cite[Theorem 2.3.1]{Asok2023}. In particular, if $f$ is transversal, there is a commuting square 
\begin{equation}\label{eq:NatOfPurity}
    \begin{tikzcd}
      X/(X \sm Z) \dar \rar["\weq"] & \Th(N_Z X) \dar \\
      X' / (X/ \sm Z') \rar["\weq"] & \Th(N_{Z'} X')
    \end{tikzcd}
  \end{equation}
where the horizontal maps are equivalences, and the vertical maps are induced by $f$. In what follows, we will refer to the commutativity of \eqref{eq:NatOfPurity} as ``the naturality of purity.'' 

Recall from \Cref{sec:frn} that there is an $\A^1$-cofibre sequence
\begin{equation}\label{eq:StiefelCof}
  V_{r-1}(\A^{n-1}) \xrar{i} V_r(\A^n) \to \Sigma^{2n-1,n} V_{r-1}(\A^{n-1})_+
\end{equation}
We sketch the construction of \eqref{eq:StiefelCof} since it will be important to us in our proof of Lemma \ref{lem:IncBasepoint}. We refer to \cite[Section 3.2]{Peng2023} for details (see also \cite[Proposition 4.2]{Rondigs2023}). 

It will suffice to sketch the construction of the analogous $\A^1$-cofibre sequence 
\[
  V_{r-1}'(\A^{n-1}) \xrar{i'} V_r'(\A^n) \to \Sigma^{2n-1,n} V_{r-1}'(\A^{n-1})_+,
\]
which is $\A^1$-equivalent to \eqref{eq:StiefelCof} (see \Cref{sec:Stiefel}).

In this appendix, we will regard $V_r'(\A^n)$ as the $k$-scheme given by full-rank $n \times r$ matrices. Let $\{ a_{i,j} \mid 1 \leq i \leq n, 1 \leq j \leq r \}$ be matrix coordinates for $V_r'(\A^n)$. We define $Z$ to be the closed subscheme of $V_r'(\A^n)$ given by
\begin{equation}\label{eq:ZDef}
  Z:= \{a_{1,r} = a_{2,r} = \cdots = a_{n-1,r} = 0\} \hookrightarrow V_r'(\A^n).
\end{equation}
In other words, $Z$ consists of those full-rank $n \times r$ matrices of the form 
\[
  \begin{bmatrix}
    * & \cdots & * & 0\\
    \vdots & \ddots &  & \vdots \\
    * & & * & 0 \\
    * & \cdots & * & * 
  \end{bmatrix}. 
\]
where $*$ denotes a possibly nonzero entry. There is a naive $\A^1$-deformation retraction of $Z$ onto the closed subscheme 
\[ 
  \{a_{1,r} = a_{2,r} = \cdots = a_{n-1,r} = a_{n,1} = \cdots = a_{n,r-1} = 0\} \iso V_{r-1}'(\A^{n-1}) \times \G_m.
\]
We may view $V_{r-1}'(\A^{n-1}) \times \A^{n-1}$ as a closed subscheme of $V_r'(\A^n)$ via the inclusion
\begin{align*}
V_{r-1}'(\A^{n-1}) \times \A^{n-1} &\hookrightarrow V_r'(\A^n) \\
  (A, (x_1, \dots, x_{n-1})) &\mapsto 
  \left[ 
    \begin{array}{ccc|c}
      & & & x_1 \\
      & A & & \vdots \\
      & & & x_{n-1} \\\hline
      0 & \cdots & 0 & 1
    \end{array}
  \right].
\end{align*}
It is shown in \cite[Section 3.2]{Peng2023} that the inclusion 
\begin{equation}\label{eq:CompZEq}
  V_{r-1}'(\A^{n-1}) \times (\A^{n-1}\sm \bar{0}) \hookrightarrow V_r'(\A^n) \sm Z
\end{equation}
is an $\A^1$-equivalence. Since $Z \hookrightarrow V_r'(\A^n)$ is of codimension $n-1$ and has a trivial normal bundle, the associated purity $\A^1$-cofibre sequence is
\[
  V_r'(\A^n) \sm Z \to V_r'(\A^n) \to \Sigma^{2n-2,n-1} (V_{r-1}'(\A^{n-1}) \times \G_m)_+.
\]
The inclusions provide a map of smooth pairs. 
\begin{equation}\label{eq:SmPair}
  (V_{r-1}'(\A^{n-1}) \times \A^{n-1}, Z \cap (V_{r-1}'(\A^{n-1}) \times \A^{n-1})) \to (V_r'(\A^n), Z)
\end{equation}
where the intersection $Z \cap (V_{r-1}'(\A^{n-1}) \times \A^{n-1}) = V_{r-1}'(\A^{n-1}) \times \bar{0}$ is transverse. Moreover, the inclusion
\[
  Z \cap (V_{r-1}'(\A^{n-1}) \times \A^{n-1}) \to Z
\]
is, up to homotopy, given by 
\[
  \id \times 1: V_{r-1}'(\A^{n-1}) \times \Spec k \to V_{r-1}'(\A^{n-1}) \times \G_m 
\]
Putting everything together, \eqref{eq:SmPair} induces a map of purity cofibre sequences 
\begin{equation}\label{eq:Purity1}
  \begin{tikzcd}[column sep = small]
    V_{r-1}'(\A^{n-1}) \times (\A^{n-1} \setminus \bar{0}) \rar \dar["\weq"] & V_{r-1}'(\A^{n-1}) \times \A^{n-1} \rar \dar & \Sigma^{2n-2,n-1} (V'_{r-1}(\A^{n-1}) \times \Spec k)_+ \dar["\Sigma^{2n-2,n-1}(\id \times 1)_+"] \\
    V_r'(\A^n) \sm Z \rar & V_r'(\A^n) \rar & \Sigma^{2n-2,n-1} (V_{r-1}'(\A^{n-1}) \times \G_m)_+
  \end{tikzcd}
\end{equation}
where the right vertical map is as indicated by the naturality of homotopy purity. Since the rows of \eqref{eq:Purity1} are $\A^1$-cofibre sequences and the left vertical map is an $\A^1$-equivalence, the right square is a pushout (after $\A^1$-localization). It follows that $\cof_{\A^1}(i')$ is equivalent to the $\A^1$-cofibre of the right vertical map in \eqref{eq:Purity1}, which is
\[
  \Sigma^{2n-2,n-1} (V_{r-1}'(\A^{n-1})_+ \wedge \G_m) \homeq \Sigma^{2n-1,n} V_{r-1}'(\A^{n-1})_+,
\]
establishing the $\A^1$-cofibre sequence \eqref{eq:StiefelCof}.

We now turn to the proof of Lemma \ref{lem:IncBasepoint}. Recall that Lemma \ref{lem:IncBasepoint} asserts that the map $\psi: S^{2n-1,n} \to \Sigma^{2n-1,n}V_{r-1}(\A^{n-1})_+$ obtained by taking $\A^1$-cofibres of the horizontal maps in 
\[
  \begin{tikzcd}
    \Sigma^{1,1} \tilde{\P}_{n-r}^{n-2} \dar["f_{r-1}^{n-1}"] \rar["\Sigma^{1,1} \tilde{\imath}"] & \Sigma^{1,1} \tilde{\P}_{n-r}^{n-1} \dar["f_r^n"] \\
    V_{r-1}(\A^{n-1}) \rar["i"] & V_r(\A^n)
  \end{tikzcd}
\]
is given by $\Sigma^{2n-1,n}(-)_+$ applied to the inclusion of the basepoint $\Spec k \to V_{r-1}(\A^{n-1})$.

We denote the restriction of $f_n': \tilde{\P}^{n-1} \times \G_m \to \GL_n$ to any subscheme of $\tilde{\P}^{n-1} \times \G_m$ also by $f_n'$. Our proof is presented through a series of reductions. The idea is to reduce to the case of $r=n$, then lift the diagram \eqref{eq:Purity1} to $\tilde{\P}^{n-1} \times (\G_m \sm 1)$ via $f_n'$ so that we can describe maps of Thom spaces of normal bundles using the naturality of the purity equivalence. One difficulty is that the map $f_n'$ is not transverse to the subscheme $Z \hookrightarrow \GL_n$ (see \eqref{eq:ZDef}). However, the restriction of $f_n'$ to $\tilde{\P}^{n-1} \times (\G_m \sm 1)$ is transverse to $Z$. 

In the case that the base field $k$ is perfect and has finite $2$-\'etale cohomological dimension, a much simpler proof of Lemma \ref{lem:IncBasepoint} can be obtained using the cohomology calculation of \cite[Theorem 18]{Williams2012} and the conservativity theorem of \cite[Theorem 16]{Bachmann2018}. We provide a geometric proof that can be easily adapted to an arbitrary base, though we will restrict our attention to the base being a field. 

We begin by reducing the proof of Lemma \ref{lem:IncBasepoint} to the case of $r=n$. It follows from the naturality of the construction of \eqref{eq:StiefelCof} and of the purity equivalence that the map 
\[ \Sigma^{2n-1,n} (\GL_{n-1})_+ \to \Sigma^{2n-1,n} V_{r-1}(\A^{n-1})_+ \] 
induced by taking $\A^1$-cofibres of the horizontal maps in the commuting square
  \[
    \begin{tikzcd}
      \GL_{n-1} \rar["i"] \dar["p"] & \GL_n \dar["p"] \\
      V_{r-1}(\A^{n-1}) \rar["i"] & V_{r}(\A^n)
    \end{tikzcd}
  \]
is $\Sigma^{2n-1,n} p_+$. Since $p$ is a pointed map, to prove Lemma \ref{lem:IncBasepoint} it suffices to treat the case of $r=n$.

Let $Y$ denote the (scheme-theoretic) fibre over the $k$-rational point $[0: \cdots : 0 :1]$ of the map $\pi: \tilde{\P}^{n-1} \to \P^{n-1}$ of Remark \ref{rem:ProjWE}. The fibre $Y$ is thus isomorphic to the affine space $\A^{n-1}$, and the $E$-rational points of $Y$, for $E/k$ a field extension, consists of pairs $(\Span\{e_n\},W)$ where $e_n \in E^n$ is the $n$th standard basis vector, $W \subseteq E^n$ is a hyperplane, and $\Span\{e_n\} + W = E^n$ (see Remark \ref{rem:FPoints}). 

There is an $\A^1$-deformation retraction of $Y$ onto the $k$-rational point 
\[  p_0 := (\Span \{e_n\}, \Span \{e_1, \dots, e_{n-1}\}). \]
Moreover, the inclusion $\tilde{\imath}: \tilde{\P}^{n-2} \hookrightarrow \tilde{\P}^{n-1} \sm Y$ is an $\A^1$-equivalence. Being the fibre of $\pi$ over a $k$-rational point, the subscheme $Y$ has codimension $n-1$ and a trivial normal bundle in $\tilde{\P}^{n-1}$. The purity $\A^1$-cofibre sequence associated with the closed inclusion $Y \times \G_m \hookrightarrow \tilde{\P}^{n-1} \times \G_m$ is thus 
\[
  \tilde{\P}^{n-2} \times \G_m \xrar{\tilde{\imath} \times \id} \tilde{\P}^{n-1} \times \G_m \to \Sigma^{2n-2,n-1}(p_0\times \G_m)_+.
\]
Similarly, there is a purity $\A^1$-cofibre sequence
\[
  \tilde{\P}^{n-2} \times (\G_m \sm 1) \xrar{\tilde{\imath} \times \id} \tilde{\P}^{n-1} \times (\G_m \sm 1) \to \Sigma^{2n-2,n-1}(p_0\times (\G_m \sm 1))_+.
\]

There are maps of $\A^1$-cofibre sequences
\[
  \begin{tikzcd}
    \tilde{\P}^{n-2} \times (\G_m \sm 1) \dar[hook] \rar["\tilde{\imath} \times \id"] & \tilde{\P}^{n-1} \times (\G_m \sm 1) \dar[hook] \rar & \Sigma^{2n-2,n-1}(p_0\times (\G_m \sm 1))_+ \dar["\Sigma^{2n-2,n-1}(\id \times \jmath)_+"] \ar[dd, bend left, "\psi''", start anchor = south east, end anchor = north east] \\
    \tilde{\P}^{n-2} \times \G_m \rar["\tilde{\imath} \times \id"] \dar["f_{n-1}'"] & \tilde{\P}^{n-1} \times \G_m \dar["f_n'"] \rar & \Sigma^{2n-2,n-1}(p_0\times \G_m)_+ \dar["\psi'"] \\
    \GL_{n-1} \rar["i"] & \GL_n \rar & \Sigma^{2n-2,n-1}(\GL_{n-1})_+ \wedge \G_m
  \end{tikzcd}
\]
where we have let $\psi',\psi''$ denote induced maps of $\A^1$-cofibres, and $\jmath: \G_m \sm 1 \hookrightarrow \G_m$ is the inclusion.

In particular, we have a commuting triangle
\begin{equation}\label{eq:psi'psi''}
  \begin{tikzcd}[row sep = large]
    \Sigma^{2n-2,n-1}(p_0 \times (\G_m \sm 1))_+ \dar["\Sigma^{2n-2,n-1}(\id \times \jmath)_+"'] \ar[dr,"\psi''"] & \\
    \Sigma^{2n-2,n-1}(p_0 \times \G_m)_+ \rar["\psi'"] & \Sigma^{2n-2,n-1} (\GL_{n-1}) \wedge \G_m \: .
  \end{tikzcd}
\end{equation}
The following lemma reduces Lemma \ref{lem:IncBasepoint} to an investigation of the map $\psi'$.

\begin{lemma}\label{lem:psi'Impliespsi}
  Suppose $\psi'$ is (up to homotopy) given by $\Sigma^{2n-2,n-1}(-)$ applied to the composition of pointed maps
  \[
    (p_0 \times \G_m)_+ \xrar{(I_{n-1} \times \id)_+} (\GL_{n-1} \times \G_m)_+ \to (\GL_{n-1})_+ \wedge \G_m
  \]
  where the last map is the canonical collapse. Then Lemma \ref{lem:IncBasepoint} holds.
\end{lemma}
\begin{proof}
  The map $\psi$ is induced by $\psi'$, so that
  \begin{equation}\label{eq:PsiPsi'}
    \begin{tikzcd}[row sep = large]
      \Sigma^{2n-2,n-1}(p_0 \times \G_m)_+ \dar \ar[dr,"\psi'"] & \\
      \Sigma^{2n-2,n-1}(p_0)_+ \wedge \G_m \rar["\psi"] & \Sigma^{2n-2,n-1}(\GL_{n-1})_+ \wedge \G_m
    \end{tikzcd}  
  \end{equation}
  commutes. Moreover, the naturality of purity with respect to the inclusion of smooth pairs
  \[
    (\tilde{\P}^{n-1} \times 1, Y \times 1) \to (\tilde{\P}^{n-1} \times \G_m, Y \times \G_m),
  \]
  implies that the vertical map in \eqref{eq:PsiPsi'} is induced by the canonical collapse 
  \[
    p_0 \times \G_m \to (p_0)_+ \wedge \G_m.
  \]
  The result follows.
\end{proof}
Next, we reduce to an investigation of the map $\psi''$. To this end, we will need the following lemma. 
\begin{lemma}\label{lem:jstar}
  Let $\sh{X}$ be a motivic space and $p,q$ nonnegative integers with $p \ge q$. The map 
  \[
    [\Sigma^{p,q}(\G_m)_+,\sh{X}] \to [\Sigma^{p,q}(\G_m \sm 1)_+,\sh{X}]
  \]
  induced by pullback along $\Sigma^{p,q}\jmath_+$ in $\cat{H}(k)_*$ is injective. 
\end{lemma}
\begin{proof}
  Consider the pushout square
  \[
    \begin{tikzcd}
      \G_m \sm 1 \ar[dr,phantom,very near end, "\ulcorner"] \dar[hook] \rar["\jmath"] & \G_m \dar[hook] \\
      \A^1 \sm 1 \rar[hook] & \A^1 
    \end{tikzcd}
  \]
  arising from the Zariski cover of $\A^1$ by $\A^1 \sm 1$ and $\G_m$. The cofibres of the horizontal maps are both $\A^1$-equivalent to $\P^1$, and the induced map between them is an $\A^1$-equivalence by the naturality of purity. Hence, the second map in the $\A^1$-cofibre sequence
  \[
    (\G_m \sm 1)_+ \xrar{\jmath_+} (\G_m)_+ \to \P^1
  \]
  is null. After applying the functor $\Sigma^{p,q}(-)$ followed by $[-,\sh{X}]$, one obtains an exact sequence 
  \[
    [\Sigma^{p,q}\P^1,\sh{X}] \xrar{0} [\Sigma^{p,q}(\G_m)_+,\sh{X}] \xrar{\Sigma^{p,q}\jmath_+^*} [\Sigma^{p,q}(\G_m \sm 1)_+,\sh{X}],
  \]
  and we conclude.
\end{proof}
We may now prove the following reduction. 

\begin{lemma}\label{lem:psi''Impliespsi}
  Suppose $\psi''$ is (up to homotopy) given by $\Sigma^{2n-2,n-1}(-)$ applied to the composition of pointed maps
  \[
    (p_0 \times (\G_m \sm 1))_+ \xrar{(\id \times \jmath)_+} (p_0 \times \G_m)_+ \xrar{(I_{n-1} \times \id)_+} (\GL_{n-1} \times \G_m)_+ \to (\GL_{n-1})_+ \wedge \G_m,
  \]
  where the last map is the canonical collapse. Then Lemma \ref{lem:IncBasepoint} holds.
\end{lemma}
\begin{proof}
  Lemma \ref{lem:jstar} and the commutativity of \eqref{eq:psi'psi''} imply that the supposition of Lemma \ref{lem:psi'Impliespsi} holds. We conclude by Lemma \ref{lem:psi'Impliespsi}.
\end{proof}

We now turn to establishing the supposition of Lemma \ref{lem:psi''Impliespsi}. Recall from \eqref{eq:ZDef} that $Z \hookrightarrow \GL_n$ is the closed subscheme of codimension $n-1$ defined by setting the coordinates $a_{1,n}, \dots, a_{n-1,n}$ equal to $0$. Let $T$ denote the (scheme-theoretic) pullback of the diagram 
\[
  \begin{tikzcd}
    T \ar[dr,phantom, very near start,"\lrcorner"] \ar[d] \ar[r] & \tilde{\P}^{n-1} \times (\G_m \sm 1) \dar["f_n'"] \\
    Z \rar[hook] & \GL_n
  \end{tikzcd}
\]
which we identify with a closed subscheme of $\tilde{\P}^{n-1} \times (\G_m \sm 1)$. The subscheme $T$ consists of two smooth, closed, disjoint components $T_1,T_2$, each of codimension $n-1$, whose $E$-rational points, for $E/k$ a field extension, are 
\[
  T_1(E) = \{ (L,W,\lambda) \in \tilde{\P}^{n-1}(E) \times (E^\times \sm 1) \mid e_n \in L \}, \quad T_2(E) = \{ (L,W,\lambda) \in \tilde{\P}^{n-1}(E) \times (E^\times \sm 1) \mid e_n \in W \},
\]
respectively. In particular, the map $f_n': \tilde{\P}^{n-1} \times (\G_m \sm 1) \to \GL_n$ is transverse to $Z \hookrightarrow \GL_n$. Note also that $T_1 = Y \times (\G_m \sm 1)$. Denote the Thom spaces of the normal bundles associated with the inclusion of $T$ and $T_i$ (for $i=1,2$) in $\tilde{\P}^{n-1} \times (\G_m \sm 1)$ by $\Th(N_T)$ and $\Th(N_{T_i})$, respectively. There are $\A^1$-equivalences
\[
  \Th(N_T) \homeq \Th(N_{T_1}) \vee \Th(N_{T_2}) \homeq \Sigma^{2n-2,n-1}(p_0\times (\G_m \sm 1))_+ \vee \Th(N_{T_2}).
\]
The map of smooth pairs
\[
  (\tilde{\P}^{n-1} \times (\G_m \sm 1) \sm T_1, T_2)  \to (\tilde{\P}^{n-1} \times (\G_m \sm 1), T)
\]
given by inclusion is transverse, so by naturality of the purity equivalence, it induces a map of purity $\A^1$-cofibre sequences
\[
  \begin{tikzcd}[column sep = small]
    (\tilde{\P}^{n-1} \times (\G_m \sm 1)) \sm T \rar \dar[equals] & (\tilde{\P}^{n-1} \times (\G_m \sm 1)) \sm T_1 \rar \dar[hook] & \Th(N_{T_2}) \dar["\inc_2"] \\
    (\tilde{\P}^{n-1} \times (\G_m \sm 1)) \sm T \rar & \tilde{\P}^{n-1} \times (\G_m \sm 1) \rar & \Sigma^{2n-2,n-1}(p_0\times (\G_m \sm 1))_+ \vee \Th(N_{T_2})
  \end{tikzcd}
\]
where the map denoted $\inc_2$ is the inclusion of the second summand. 

We conclude with the following lemma, which establishes Lemma \ref{lem:IncBasepoint} by Lemma \ref{lem:psi''Impliespsi}.

\begin{lemma}\label{lem:Psi''}
  The map $\psi''$ is (up to homotopy) given by $\Sigma^{2n-2,n-1}(-)$ applied to the composition of pointed maps
  \[
    (p_0 \times (\G_m \sm 1))_+ \xrar{(\id \times \jmath)_+} (p_0 \times \G_m)_+ \xrar{(I_{n-1} \times \id)_+} (\GL_{n-1} \times \G_m)_+ \to (\GL_{n-1})_+ \wedge \G_m,
  \]
  where the last map is the canonical collapse.
\end{lemma} 
\begin{proof}
  We consider the diagram
  \begin{equation}\label{eq:BigPurity}
    \adjustbox{scale=0.85}{
    \begin{tikzcd}
      &[-12ex] (\tilde{\P}^{n-1} \times (\G_m \sm 1)) \ar[dl,equals]\sm T \ar[dd,"\dagger",near end] \ar[rr] &[-10ex] &[-12ex] (\tilde{\P}^{n-1} \times (\G_m \sm 1)) \sm T_1 \ar[rr] \ar[dl,hook] \ar[dd,"\dagger", near end] &[-13ex] &[-19ex]  \Th(N_{T_2}) \ar[dd] \ar[dl,"\inc_2"] \\
      (\tilde{\P}^{n-1} \times (\G_m \sm 1)) \sm T \ar[dd] \ar[rr,crossing over] & & \tilde{\P}^{n-1} \times (\G_m \sm 1) \ar[rr,crossing over] & & \Sigma^{2n-2,n-1}(p_0 \times (\G_m \sm 1))_+ \vee \Th(N_{T_2}) & \\
      & \GL_{n-1} \times (\A^{n-1} \sm \bar{0}) \ar[dl,hook,"\weq"] \ar[rr] & & \GL_{n-1} \times \A^{n-1} \ar[dl] \ar[rr] &&  \Sigma^{2n-2,n-1} (\GL_{n-1} \times \Spec k)_+ \ar[dl,"\Sigma^{2n-2,n-1}(\id \times 1)_+"] \\
      \GL_n \sm Z \ar[rr] & & \GL_n \ar[rr] \arrow[from=2-3, to=4-3, crossing over] && \Sigma^{2n-2,n-1}(\GL_{n-1} \times \G_m)_+ \arrow[from=2-5, to=4-5, crossing over] &
    \end{tikzcd}
    }
  \end{equation}
  where the bottom face is \eqref{eq:Purity1} (when $r=n$), the rows are $\A^1$-cofibre sequences arising from purity, the vertical maps are all induced by $f_n'$, and the maps marked with $\dagger$ only exist up to homotopy.
  
  
  The inclusion 
  \[  
    \tilde{\P}^{n-2} \times (\G_m \sm 1) \hookrightarrow (\tilde{\P}^{n-1} \times (\G_m \sm 1)) \sm T_1 = (\tilde{\P}^{n-1} \sm Y) \times (\G_m \sm 1)
  \]
  is an $\A^1$-equivalence, so the rightmost cube of \eqref{eq:BigPurity} is equivalent to 
  \begin{equation}\label{eq:BigPurity2}
    \begin{tikzcd}[column sep = -4ex]
        & \tilde{\P}^{n-2} \times (\G_m \sm 1) \ar[dd,"f_{n-1}'",near end] \ar[dl,"\tilde{\imath} \times \id"] \ar[rr] &&[-10ex] \Th(N_{T_2}) \ar[dl,"\inc_2"] \ar[dd] \\
        \tilde{\P}^{n-1} \times (\G_m \sm 1) \ar[dd,"f_n'"] \ar[rr,crossing over] && \Sigma^{2n-2,n-1}(p_0 \times (\G_m \sm 1))_+ \vee \Th(N_{T_2}) & \\
        & \GL_{n-1} \ar[dl,"i"] \ar[rr] && \Sigma^{2n-2,n-1} (\GL_{n-1} \times \Spec k)_+ \ar[dl,"\Sigma^{2n-2,n-1}(\id \times 1)_+"] \\
        \GL_n \ar[rr] && \Sigma^{2n-2,n-1} (\GL_{n-1} \times \G_m)_+ \ar[from=2-3,to=4-3,crossing over] & 
    \end{tikzcd}
  \end{equation}
  The top and bottom squares in \eqref{eq:BigPurity2} are pushout squares (after $\A^1$-localization) since the equality and the inclusion $\GL_{n-1} \times (\A^{n-1} \sm \bar{0}) \hookrightarrow \GL_n \sm Z$ in \eqref{eq:BigPurity} are $\A^1$-equivalences.

  The map 
  \[
    \psi'': \Sigma^{2n-2,n-1} (p_0 \times (\G_m \sm 1))_+ \to \Sigma^{2n-2,n-1} (\GL_{n-1})_+ \wedge \G_m.
  \]
  is thus equivalent to the map induced by taking $\A^1$-cofibres of the horizontal maps in
  \begin{equation}\label{eq:ThomSpaces}
    \begin{tikzcd}[column sep = 10ex]
      \Th(N_{T_2}) \rar["\inc_2"] \dar & \Sigma^{2n-2,n-1}(p_0 \times (\G_m \sm 1))_+ \vee \Th(N_{T_2}) \dar \\
      \Sigma^{2n-2,n-1}(\GL_{n-1} \times \Spec k)_+ \rar["\Sigma^{2n-2,n-1}(\id \times 1)_+"] & \Sigma^{2n-2,n-1}(\GL_{n-1} \times \G_m)_+
    \end{tikzcd}
  \end{equation}
  (the right-hand face of \eqref{eq:BigPurity2}). The right vertical map of \eqref{eq:ThomSpaces} is, on the first factor, given by $\Sigma^{2n-2,n-1}(-)_+$ applied to
  \begin{equation}\label{eq:MapOnp0}
    p_0 \times (\G_m \sm 1) \hookrightarrow p_0 \times \G_m \xrar{I_{n-1} \times \id} \GL_{n-1} \times \G_m.
  \end{equation}
  This follows from the naturality of purity and the fact that  
  \[
    f_n': X_1=Y \times (\G_m \sm 1) \to Z
  \]
  is, up to homotopy, given by \eqref{eq:MapOnp0}. Assembling these facts, we obtain Lemma \ref{lem:Psi''} by taking $\A^1$-cofibres of the horizontal maps in \eqref{eq:ThomSpaces}.
\end{proof}

\printbibliography

\end{document}